\documentclass[11pt, a4paper]{article}

\usepackage{amssymb,amsmath, wasysym, graphicx, epsfig, setspace, wrapfig, comment}
\usepackage{latexsym,dsfont,amsfonts,amsmath,amssymb,amsthm}
\usepackage{pstricks,epsfig,psfrag,color}
\usepackage{bbm}
\usepackage[show]{ed}
\usepackage{soul}
\usepackage{mathrsfs}   
\usepackage{tikz}
\usepackage{upgreek}
\usepackage{algpseudocode,algorithm} 
\algdef{SE}[DOWHILE]{Do}{doWhile}{\algorithmicdo}[1]{\algorithmicwhile\ #1}

\usepackage{mathtools}

\newcommand{\pd}[3]{\frac{\partial ^{#1} #2}{\partial #3}}

\newcommand*{\grad}{\nabla}
\newcommand{\Lap}{\Delta}

\newcommand{\Ph}{\calP_h} 

\newcommand{\bs}[1]{\boldsymbol{#1}}
\newcommand{\nrm}[2]{\ensuremath{\|#1\|_{#2}}}

\newcommand{\pdy}{\partial_{\bsy}}

\newcommand{\cost}{\mathrm{cost}}

\newcommand{\diam}{\mathrm{diam}}

\newcommand{\N}[0]{\mathbb{N}}

\newcommand{\R}[0]{\mathbb{R}}

\newlength\figureheight
\newlength\figurewidth

\makeatletter
\newcommand{\BIG}{\bBigg@{3}}
\newcommand{\vast}{\bBigg@{4}}
\newcommand{\Vast}{\bBigg@{5}}
\makeatother

\newcommand{\Qhat}[0]{\ensuremath{\widehat{Q}}}
\newcommand{\Vhat}[0]{\ensuremath{\widehat{\bbV}}}

\newcommand{\Qml}[0]{\ensuremath{Q^\mathrm{ML}}}
\newcommand{\Qhatml}[0]{\ensuremath{\Qhat^\mathrm{ML}}}

\newtheorem{theorem}{Theorem}[section]
\newtheorem{lemma}{Lemma}[section]
\newtheorem{corollary}{Corollary}[section]
\newtheorem{proposition}[theorem]{Proposition}
\newtheorem{assumption}{Assumption A$\!\!\!$}
\theoremstyle{definition}
\newtheorem{remark}{Remark}[section]
\theoremstyle{plain}

\numberwithin{equation}{section}

\newcommand{\bsbeta}{{\boldsymbol{\beta}}}
\newcommand{\bsDelta}{{\boldsymbol{\Delta}}}

\newcommand{\bspsi}{{\boldsymbol{\psi}}}

\newcommand{\bse}{{\boldsymbol{e}}}

\newcommand{\bsgamma}{{\boldsymbol{\gamma}}}

\newcommand{\bsnu}{{\boldsymbol{\nu}}}
\newcommand{\bsm}{{\boldsymbol{m}}}
\newcommand{\bstau}{{\boldsymbol{\tau}}}

\newcommand{\bsh}{{\boldsymbol{h}}}

\newcommand{\bsk}{{\boldsymbol{k}}}

\newcommand{\bst}{{\boldsymbol{t}}}

\newcommand{\bsx}{{\boldsymbol{x}}}

\newcommand{\bsy}{{\boldsymbol{y}}}

\newcommand{\bsz}{{\boldsymbol{z}}}

\newcommand{\bszero}{{\boldsymbol{0}}}

\newcommand{\rd}{\,\mathrm{d}}

\newcommand{\bbE}{\mathbb{E}}
\newcommand{\bbV}{\mathbb{V}}
\newcommand{\calA}{\mathcal{A}}

\newcommand{\calI}{\mathcal{I}}

\newcommand{\calM}{\mathcal{M}}

\newcommand{\calF}{\mathcal{F}}
\newcommand{\calG}{\mathcal{G}}

\newcommand{\calO}{\mathcal{O}}

\newcommand{\calP}{\mathcal{P}}
\newcommand{\calW}{\mathcal{W}}

\def\R{\mathbb{R}}

\newcommand{\setu}{{\mathrm{\mathfrak{u}}}}
\newcommand{\setv}{{\mathrm{\mathfrak{v}}}}

\newcommand{\mask}[1]{{}}

\definecolor{darkred}{RGB}{139,0,0}
\definecolor{darkgreen}{RGB}{0,100,0}
\definecolor{darkmagenta}{RGB}{170,0,120}
\definecolor{darkpurple}{RGB}{110,0,180}
\definecolor{darkblue}{RGB}{40,0,200}
\definecolor{darkbrown}{rgb}{0.75,0.40,0.15}
\definecolor{grey}{RGB}{59,61,63}

\newcommand{\be}{\begin{equation}}
\newcommand{\ee}{\end{equation}}
\newcommand{\bea}{\begin{eqnarray}}
\newcommand{\eea}{\end{eqnarray}}
\newcommand{\beas}{\begin{eqnarray*}}
\newcommand{\eeas}{\end{eqnarray*}}

\def\r2p{{\sqrt{2\pi}}}

\newcommand*{\CI}{\ensuremath{C_\mathrm{I}}}
\newcommand*{\CII}{\ensuremath{C_\mathrm{II}}}
\newcommand*{\CIII}{\ensuremath{C_\mathrm{III}}}

\newcommand*{\Chat}{\ensuremath{\widehat{C}}}

\newcommand*{\evalueLap}{\ensuremath{\chi_1}}
\newcommand*{\amin}{\ensuremath{a_{\min}}}
\newcommand*{\amax}{\ensuremath{a_{\max}}}

\newcommand*{\betabar}{\ensuremath{\overline{\beta}}}
\newcommand*{\betahat}{\ensuremath{\widehat{\beta}}}
\newcommand*{\bsbetabar}{\ensuremath{\overline{\bsbeta}}}
\newcommand*{\bsbetahat}{\ensuremath{\widehat{\bsbeta}}}
\newcommand*{\lambdabar}{\ensuremath{\overline{\lambda}}}
\newcommand*{\ubar}{\ensuremath{\overline{u}}}

\newcommand*{\hsuff}{\ensuremath{\overline{h}}}

\newcommand{\indx}{{\calF}}
\newcommand{\Ws}[0]{\ensuremath{\calW_{s, \bsgamma}}}

\newcommand{\bigO}{\mathcal{O}}

\graphicspath{{.}{./Figures/}}

\title{Multilevel quasi-Monte Carlo for random elliptic eigenvalue problems I: Regularity and error analysis}

\date{\today}

\makeatletter
\let\@fnsymbol\@arabic
\makeatother

\author{Alexander D. Gilbert\footnotemark[1]
				\and
             Robert Scheichl\footnotemark[2]
             }

\begin{document}

\maketitle

\footnotetext[1]{School of Mathematics and Statistics, University of New South Wales, 
                           Sydney NSW 2052, Australia.\\
                           \texttt{alexander.gilbert@unsw.edu.au}
                           }
\footnotetext[2]{Institute for Applied Mathematics \& Interdisciplinary Centre  for Scientific Computing,
								 Universit\"at Heidelberg, 
                          69120 Heidelberg, Germany and
                          Department of Mathematical Sciences, University of Bath, Bath BA2 7AY UK.\\
                          \texttt{r.scheichl@uni-heidelberg.de}
                          }

\begin{abstract}
{Stochastic PDE eigenvalue problems are useful models for quantifying the uncertainty in several
applications from the physical sciences and engineering, e.g.,
structural vibration analysis, the criticality of a nuclear reactor or 
photonic crystal structures.
In this paper we present a multilevel quasi-Monte Carlo (MLQMC) method for
approximating the expectation of the minimal eigenvalue of an elliptic eigenvalue problem
with coefficients that are given as a series expansion of countably-many stochastic parameters.
The MLQMC algorithm is based on a hierarchy of discretisations of the spatial domain 
and truncations of the dimension of the stochastic parameter domain.
To approximate the expectations,
randomly shifted lattice rules are employed.
This paper is primarily dedicated to giving a rigorous analysis of the 
error of this algorithm.
A key step in the error analysis requires bounds on the mixed derivatives of
the eigenfunction with respect to both the stochastic and spatial variables simultaneously.
Under stronger smoothness assumptions on the parametric dependence,
our analysis also extends to multilevel higher-order quasi-Monte Carlo rules.
An accompanying paper [Gilbert and Scheichl, 2022], 
focusses on practical extensions of the MLQMC algorithm to improve efficiency,
and presents numerical results.}
\end{abstract} 

\section{Introduction}
Consider the following elliptic eigenvalue problem (EVP)
\begin{align}
\label{eq:evp}
\nonumber 
-\nabla\cdot\big(a(\bsx, \bsy)\,\nabla u(\bsx, \bsy)\big) 
+ b(\bsx, \bsy)\,u(\bsx, \bsy)
&= 
\lambda(\bsy) \,c(\bsx,\bsy ) \,u(\bsx, \bsy), \ \  &\text{for } \bsx \in D,
\\
u(\bsx, \bsy) \,&=\, 0 \quad &\text{for }\bsx \in \partial D,
\end{align}
where the differential operator $\nabla$ is with respect to the 
{\em physical variable} $\bsx$, which belongs to a bounded, convex domain 
$D \subset \R^d$ ($d = 1, 2, 3$), 
and where the {\em stochastic parameter}
\begin{equation*}
\bsy \,=\, (y_j)_{j \in \N} \in \Omega \coloneqq [-\tfrac{1}{2}, \tfrac{1}{2}]^\N ,
\end{equation*}
is an infinite-dimensional vector of independently and identically distributed (i.i.d.)
uniform random variables on $[-\tfrac{1}{2}, \tfrac{1}{2}]$.

The dependence of the coefficients on the stochastic parameters carries through
to the eigenvalues $\lambda(\bsy)$, and corresponding eigenfunctions 
$u(\bsy) \coloneqq u(\cdot, \bsy)$, and as such, in this paper we are interested
in computing statistics of the eigenvalues and of linear functionals of the 
corresponding eigenfunction.
In particular, we would like to compute the expectation, with respect to the
countable product of uniform densities, of the smallest eigenvalue $\lambda$, 
which is an infinite-dimensional integral defined as
\[
\bbE_\bsy[\lambda] \,=\, \int_{\Omega} \lambda(\bsy) \, \rd\bsy
\,\coloneqq\, \lim_{s \to \infty} \int_{[-\frac{1}{2}, \frac{1}{2}]^s}
\lambda(y_1, y_2, \ldots, y_s, 0, 0 \ldots)\, \rd y_1 \rd y_2 \cdots \rd y_s.
\]

The multilevel Monte Carlo (MLMC) method \cite{Giles08a,Hein01} is a variance reduction 
scheme that has been successfully applied to many stochastic simulation problems.
When applied to stochastic PDE problems (see, e.g., \cite{BarSchwZol11,ClifGilSchTeck11}), 
the MLMC method is based on a hierarchy
of $L + 1$ increasingly fine finite element meshes $\{\mathscr{T}_\ell\}_{\ell = 0}^L$
(corresponding to a decreasing sequence of meshwidths $h_0 > h_1 > \cdots > h_L> 0$),
and an increasing sequence of truncation dimensions $s_0 < s_1 < \cdots < s_L < \infty$.
Letting the dimension-truncated FE approximation on level $\ell$ be denoted by
$\lambda_\ell \coloneqq \lambda_{h_\ell, s_\ell}$, by linearity, we can write the 
expectation on the finest level as
\begin{equation}
\label{eq:tele_sum}
\bbE_\bsy [ \lambda_L] \,=\, \bbE_\bsy[\lambda_0] + \sum_{\ell = 1}^L \bbE_\bsy[\lambda_\ell - \lambda_{\ell - 1}].
\end{equation}
Each expectation $\bbE_\bsy[ \lambda_\ell - \lambda_{\ell - 1}]$ is then approximated
by an independent Monte Carlo method.
Defining $u_\ell \coloneqq u_{h_\ell, s_\ell}$ we can write a similar telescoping sum for
$\bbE_\bsy[\calG(u_L)]$ for any linear functional $\calG(u)$.

Quasi-Monte Carlo (QMC) methods are equal-weight quadrature rules where the samples are
deterministically chosen to be well-distributed, see \cite{DKS13}.
Multilevel quasi-Monte Carlo (MLQMC) methods, whereby a QMC quadrature rule
to approximate the expectation on each level,
were first developed in \cite{GilesWater09} 
for path simulation with applications in option pricing and then later applied to 
stochastic PDE problems (e.g., \cite{KSS15,KSSSU17}).
For certain problems, MLQMC methods can be shown to converge faster than 
their Monte Carlo counterpart, and for most problems the gains from
using multilevel and QMC are complementary.

In this paper we present a rigorous analysis of the error of a MLQMC algorithm for
approximating the expectation of the smallest eigenvalue of \eqref{eq:evp}
in the case where the coefficients are given
by a Karhunen--Lo\`{e}ve type series expansion.
The main result proved in this paper is that under some common assumptions
on the summability of the terms in the coefficient expansion,
the root-mean-square error (RMSE) of a MLQMC approximation of $\bbE_\bsy[\lambda]$,
which on each level $\ell = 0, 1, \ldots, L$
uses a randomly shifted lattice rule with $N_\ell$ points, a FE discretisation with meshwidth $h_\ell > 0$
and a fixed truncation dimension,
is bounded by
\begin{equation}
\label{eq:err-intro}
\text{RMSE} \,\lesssim\, 
h_L^2 + \sum_{\ell  = 0}^L N_\ell^{-1 + \delta} h_\ell^2,
\quad \text{for } \delta > 0,
\end{equation}
with a similar result for the eigenfunction (see Theorems~\ref{thm:ML_abstract_lam},
\ref{thm:ML_abstract_G} and Remark~\ref{rem:err-params}).
This error bound is clearly better than the corresponding result for a MLMC method,
which has $N_\ell^{-1 + \delta}$ replaced by $N_\ell^{-1/2}$,
and in terms of the overall complexity
to achieve a RMSE less than some tolerance $\varepsilon > 0$ the total cost compared to a single level 
QMC approximation is reduced by a factor of $\varepsilon^{-1}$ in spatial dimensions $d \geq 2$
(see Corollary~\ref{cor:complexity}).
Under equivalent assumptions, the convergence rates in \eqref{eq:err-intro} coincide
with the rates in the corresponding error bound 
for source problems from \cite{KSS15,KSSSU17}.
Although it is not unexpected that we are able to 
obtain the same convergence rates as for source problems, the
analysis here is completely new and because of the nonlinear nature of eigenvalue problems
presents several added difficulties not encountered previously in the analysis of source problems. 
Indeed, the key intermediate step is an in-depth analysis
of the mixed regularity of the eigenfunction, simultaneously 
in both the spatial and stochastic variables.
The result, presented in Theorem~\ref{thm:regularity_y}, 
is a collection of explicit bounds on the mixed derivatives of the eigenfunction,
where the derivatives are second order with respect to the spatial variable $\bsx$ and 
arbitrarily high order with respect to the stochastic variable $\bsy$.
The proof of these bounds forms a substantial proportion of this paper, 
and requires a delicate multistage induction
argument along with a considerable amount of technical analysis 
(see Section~\ref{sec:reg} and the Appendix). 
These bounds significantly extend the previous regularity results for stochastic EVPs 
from \cite{AS12}, which  didn't give any bounds on the derivatives,
and \cite{GGKSS19}, which gave bounds that were first-order with respect to $\bsx$
and higher order with respect to $\bsy$.
Furthermore,
many other multilevel methods require similar mixed regularity bounds for their analysis,
e.g., multilevel stochastic collocation \cite{TeckJanWebGunz15}.
Hence, the bounds are of independent interest and open the door
for further research into methods for uncertainty quantification for stochastic 
EVPs.
In particular, we show how the mixed regularity bounds can be immediately applied to extend the
analysis also to multilevel quasi-Monte Carlo methods for EVPs based on 
\emph{higher-order} interlaced polynomial 
lattice rules \cite{Dick08,GodaDick15}, following the papers 
\cite{DKLeGNS14,DKLeGS16} for source problems (see Section~\ref{sec:hoqmc}).

The focus of this paper is the theoretical analysis of our MLQMC algorithm for EVPs. As such numerical results and practical details on how to efficiently implement
the algorithm will be given in a separate paper \cite{GS21b}.

EVPs provide a useful way to model problems from a diverse range of
applications, such as structural vibration analysis \cite{Thom81}, 
the nuclear criticality problem \cite{DH76,JC13,W66} and 
photonic crystal structures \cite{D99,GG12,K01,NS12}.
More recently, interest in stochastic EVPs has been driven 
by a desire to quantify the uncertainty in applications such as nuclear physics
\cite{AI10,AEHW12,W10,W13}, structural analysis \cite{ShinAst72}
and aerospace engineering \cite{QiuLyu20}.
The most widely used numerical methods for stochastic EVPs 
are Monte Carlo methods \cite{ShinAst72}.
More recently stochastic collocation methods \cite{AS12} and stochastic 
Galerkin/polynomial chaos methods \cite{GhanGhosh07,W10,W13} have been developed. 
In particular, to deal with the high-dimensionality of the parameter
space, sparse and low-rank methods have been considered, see \cite{AS12,ElmSu19,GrubHakLaak19,HakKaarLaak15,HakLaak19}. 
Additionally, the present authors (along with colleagues) 
have applied quasi-Monte Carlo methods to \eqref{eq:evp} and 
proved some key properties of the minimal eigenvalue and its corresponding eigenfunction, 
see \cite{GGKSS19,GGSS20}.

Although we consider the smallest eigenvalue, the MLQMC method and analysis
in this paper can easily be extended to any simple eigenvalue that is well-separated
from the rest of the spectrum for all parameters $\bsy$.
If the quantity of interest depends on a cluster of eigenvalues,
or on the corresponding subspace of eigenfunctions, then, in principle,  
the method in this paper could be used in conjunction with a subspace-based eigensolver.
Again, one important point for the theory would be that the eigenvalue cluster
is well-separated from the rest of the spectrum, uniformly in $\bsy$.

The structure of the paper is as follows. In Section~\ref{sec:background} we give a
brief summary of the required mathematical material.
Then in Section~\ref{sec:mlqmc-alg} we present the MLQMC algorithm along with
a cost analysis. Section~\ref{sec:reg} proves the key regularity bounds, which are then
required for the error analysis in Section~\ref{sec:err}. 
Finally, in the appendix we give the proof of the two key lemmas from Section~\ref{sec:err}.

\section{Mathematical background}
\label{sec:background}
In this section we briefly summarise the relevant material
on variational EVPs, finite element methods and
quasi-Monte Carlo methods. For further details we refer the reader
to the references indicated throughout, or  \cite{GGKSS19}.

As a start, we make the following assumptions on the coefficients, which
will ensure that the problem \eqref{eq:evp} is well-posed and admits
fast convergence rates of our MLQMC algorithm.
In particular, we assume that
all coefficients are bounded from above and below, independently
of $\bsx$ and $\bsy$. 

\begin{assumption}
\hfill
\label{asm:coeff}
\begin{enumerate}
\item\label{itm:coeff} $a$ and $b$ are of the form 
\begin{align}
\label{eq:coeff}
a(\bsx, \bsy) = a_0(\bsx) + \sum_{j = 1}^\infty y_j a_j(\bsx)
\quad \text{and} \quad
b(\bsx, \bsy) = b_0(\bsx) + \sum_{j = 1}^\infty y_j b_j(\bsx),
\end{align}
where $a_j,\ b_j \in L^\infty(D)$, for all $j \ge 0$, and $c \in L^\infty(D)$
depend on $\bsx$ but not $\bsy$.
\item\label{itm:amin} There exists $\amin > 0$ such that $a(\bsx, \bsy) \geq \amin$,
$b(\bsx, \bsy) \geq 0$ and $c(\bsx) \geq \amin$, for all $\bsx \in D$, $\bsy \in \Omega$.
\item\label{itm:summable} There exist $p, q \in (0, 1)$ such that
\begin{align*}
\sum_{j = 1}^\infty \max\big(\nrm{a_j}{L^\infty},\, \nrm{b_j}{L^\infty} \big)^p < \infty
\quad \text{and} \quad
\sum_{j = 1}^\infty \nrm{\nabla a_j}{L^\infty(D)}^q \,<\, \infty.
\end{align*}
\end{enumerate}

For convenience, we define $\amax < \infty$ so that
\begin{equation}
\label{eq:coeff_bnd}
\max \big\{ \|a(\bsy)\|_{L^\infty}, \, \|\nabla a(\bsy) \|_{L^\infty}, \,
\|b(\bsy)\|_{L^\infty}, \|c\|_{L^\infty}\big\}
\,\leq\, \amax
\quad
\text{for all } \bsy \in \Omega.
\end{equation}

\end{assumption}

\subsection{Variational eigenvalue problems}
\label{sec:var-evp}
To introduce the variational form of the PDE \eqref{eq:evp}, we let 
$V \coloneqq H^1_0(D)$, the first order Sobolev space of functions with vanishing trace,
and equip $V$ with the norm $\nrm{v}{V} \coloneqq \nrm{\nabla v }{L^2(D)}$.
The space $V$ together with its dual, which we denote by $V^*$, satisfy the well-known
chain of compact embeddings $V \subset\subset L^2(D) \subset\subset V^*$,
where the pivot space $L^2(D)$ is identified with its own dual.

For $v, w \in V$, define the inner products 
$\calA(\bsy; \cdot, \cdot), \calM(\cdot, \cdot): V \times V \to \R$ by
\begin{align*}
\calA(\bsy; w, v) \,&\coloneqq\, 
\int_D a(\bsx, \bsy) \nabla w(\bsx) \cdot \nabla v(\bsx)\rd \bsx 
+ \int_D b(\bsx, \bsy) w(\bsx) v(\bsx) \rd \bsx,
\\
\calM(w, v) \,&\coloneqq\, \int_D c(\bsx) w(\bsx) v(\bsx) \rd \bsx,
\end{align*}
and let their respective induced norms be given by 
$\nrm{v}{\calA(\bsy)} \coloneqq \sqrt{\calA(\bsy; v, v)}$
and $\nrm{v}{\calM} \coloneqq \sqrt{\calM(v, v)}$.
Further, let $\calM(\cdot, \cdot)$ also denote the duality paring on $V \times V^*$.

In the usual way, multiplying \eqref{eq:evp} by $v \in V$ and performing integration by parts
with respect to $\bsx$, we arrive at the following variational EVP, which
is equivalent to \eqref{eq:evp}.
Find $\lambda(\bsy) \in \R$, $u(\bsy) \in V$ such that
\begin{align}
\label{eq:var-evp}
\calA(\bsy; u(\bsy), v) \,&=\, \lambda(\bsy) \calM(u(\bsy), v)
\quad \text{for all } v \in V,\\
\nonumber
\nrm{u(\bsy)}{\calM} \,&=\, 1.
\end{align}

The classical theory for symmetric EVPs (see, e.g., \cite{BO91})
ensures that the variational EVP \eqref{eq:var-evp} has 
countably many strictly positive eigenvalues,
which, counting multiplicities, we label in ascending order as
\[
0 \,<\, \lambda_1(\bsy) \,\leq\, \lambda_2(\bsy) \,\leq\, \cdots.
\]
The corresponding eigenfunctions,
\[
u_1(\bsy), \ u_2(\bsy),\ \ldots,
\]
can be chosen to form a basis of $V$ that is orthonormal with respect to
the inner product $\calM(\cdot, \cdot)$, and, by \eqref{eq:var-evp}, also orthogonal with respect to 
$\calA(\bsy; \cdot, \cdot)$.

\begin{proposition}\label{eq:prop}
The smallest eigenvalue is simple for all $\bsy \in \Omega$. 
Furthermore, there exists $\rho > 0$, independent of $\bsy$, such that
\begin{equation}
\label{eq:gap}
\lambda_2(\bsy) - \lambda_1(\bsy) \,\geq\, \rho 
\quad \text{for all } \bsy \in \Omega.
\end{equation}
\end{proposition}
\begin{proof}
The Krein--Rutmann Theorem and \cite[Proposition~2.4]{GGKSS19}.
\end{proof}

Henceforth, we will let the smallest eigenvalue and its corresponding eigenfunction be simply denoted
by $\lambda = \lambda_1$ and $u = u_1$.

It is often useful to compare the eigenvalues $\lambda_k$ to the 
eigenvalues of the negative Laplacian on $D$, also with homogeneous Dirichlet boundary conditions and with respect to the standard $L^2$ inner product.
These are denoted by
\begin{equation}
\label{eq:Lap_eval}
0 \,<\, \chi_1 \,<\, \chi_2 \,\leq\, \chi_3 \,\leq\, \cdots,
\end{equation}
and will often simply be referred to as Laplacian eigenvalues
or eigenvalues of the Laplacian, without explicitly stating the domain or boundary conditions.

The following form of the Poincar\'e inequality will also be useful throughout this paper
\begin{align}\label{eq:poin}
\nrm{v}{L^2} \,\leq\, \evalueLap^{-1/2}\nrm{v}{V} \, ,
\quad \text{for } v \in V.
\end{align}
It follows by the min-max representation for the Laplacian eigenvalue $\chi_1$.

The upper and lower bounds on the coefficients \eqref{eq:coeff_bnd}, along with the
Poincar\'e inequality \eqref{eq:poin}, ensure that the $\calA(\bsy)$- and $\calM$-norms
are equivalent to the $V$- and $L^2$-norms, respectively, with
\begin{align}
\label{eq:A_equiv}
\sqrt{\amin} \nrm{v}{V} \,\leq\, &\nrm{v}{\calA(\bsy)} \, \leq\, 
\sqrt{\amax\bigg(1 + \frac{1}{\evalueLap}\bigg)}\nrm{v}{V},\\
\label{eq:M_equiv}
\sqrt{\amin} \nrm{v}{L^2} \,\leq\, &\nrm{v}{\calM} \,\leq\, \sqrt{\amax} \nrm{v}{L^2}.
\end{align}

Finally, as is to be expected, for our finite element error analysis we require 
second-order smoothness with respect to the spatial variables, 
which we characterise by the space $Z = H^2(D) \cap V$, equipped with the norm
\[
\nrm{v}{Z} \,\coloneqq\, \big( \nrm{v}{L^2}^2 + \nrm{\Delta v}{L^2}^2\big)^{1/2}.
\]
In particular, the eigenfunctions belong to $Z$, see  \cite[Proposition~2.1]{GGKSS19}.

\subsection{Stochastic dimension truncation}
The first type of approximation we make is to truncate the infinite dimensional
stochastic domain to finitely many dimensions, which, for a truncation dimension $s \in \N$, 
we do by simply setting $y_j = 0$ for all $j > s$.
The result is that the coefficients $a$ and $b$ now only depend on $s$ terms.
We define the following notation: $\bsy_s = (y_1, y_2, \ldots, y_s)$,
\[
a^s(\bsx, \bsy) \,\coloneqq\, a_0(\bsx) + \sum_{j = 1}^s y_j a_j(\bsx),
\quad
b^s(\bsx, \bsy) \,\coloneqq\, b_0(\bsx) + \sum_{j = 1}^s y_j b_j(\bsx),
\]
and
\[
\calA_s(\bsy; w, v) \,\coloneqq 
\int_D a^s(\bsx, \bsy) \nabla w(\bsx) \cdot \nabla v(\bsx) \rd \bsx
+ \int_D b^s(\bsx, \bsy)w(\bsx) v(\bsx) \rd \bsx.
\]
So that the truncated approximations, denoted by $(\lambda_s(\bsy), u_s(\bsy))$,
satisfy
\begin{equation}
\label{eq:trunc-evp}
\calA_s(\bsy; u_s(\bsy), v) \,=\, \lambda_s(\bsy) \calM(u_s(\bsy), v)
\quad \text{for all } v \in V.
\end{equation}

\subsection{Finite element methods for EVPs}
\label{sec:fem}
To begin with, we first describe the finite element (FE) spaces used to discretise
the EVP \eqref{eq:var-evp}.
Let $\{V_h\}_{h > 0}$ be a family of conforming FE spaces
of dimension $M_h$,
where each $V_h$ corresponds to a shape regular triangulation $\mathscr{T}_h$
of $D$ and the index parameter $h = \max\{\diam(\tau) : \tau \in \mathscr{T}_h\}$
is called the meshwidth.
Since we have only assumed that the domain 
$D$ is convex and $a \in W^{1, \infty}(D)$,
throughout this paper we only consider continuous, piecewise linear FE spaces.
However, under stricter conditions on the smoothness of the domain and the
coefficients, one could easily extend our algorithm to higher-order FE methods.
Furthermore, we assume that the number of FE degrees of freedom is of the order of $h^{-d}$,
so that  $M_h \eqsim h^{-d}$.
This condition is satisfied by quasi-uniform meshes and also allows for local refinement.

For $h > 0$, each $\bsy\in \Omega$ yields a FE (or discrete) EVP, 
which is formulated as: Find $\lambda_h(\bsy) \in \R$, $u_h(\bsy) \in V_h$ such that
\begin{align}
\label{eq:fe-evp}
\calA(\bsy; u_h(\bsy), v_h) \,&=\, \lambda_h(\bsy) \calM(u_h(\bsy), v_h)
\quad \text{for all } v_h \in V_h,\\
\nonumber
\nrm{u_h(\bsy)}{\calM} \,&=\, 1.
\end{align}

The discrete EVP \eqref{eq:fe-evp} has
$M_h$ eigenvalues
\[
0 \,<\, \lambda_{1, h}(\bsy) \,\leq\, \lambda_{2, h}(\bsy) \,\leq\, \cdots \,\leq\, \lambda_{M_h, h}(\bsy),
\]
and corresponding eigenfunctions
\[
u_{1, h}(\bsy),\ u_{2, h}(\bsy),\ \ldots,\ u_{M_h, h}(\bsy),
\]
which are known to converge to the first $M_h$
eigenvalues and eigenfunctions of \eqref{eq:var-evp} as $h \to 0$,
see, e.g., \cite{BO91} or \cite{GGKSS19} for the stochastic case.

From \cite[Theorem~2.6]{GGKSS19} we have the following bounds on the FE
error for the minimal eigenpair,
which we restate here because they will be used extensively in our error analysis in Section~\ref{sec:err}.

\begin{theorem}
\label{thm:fe_err}
Let $h > 0$ be sufficiently small and suppose that Assumption~A\ref{asm:coeff} holds.
Then, for all $\bsy \in \Omega$, $\lambda_h$ satisfies 
\begin{equation}
\label{eq:fe_lam}
|\lambda(\bsy) - \lambda_h(\bsy)| \,\leq\, C_{\lambda} h^2,
\end{equation}
the corresponding eigenfunction $u_h$ can be chosen such that
\begin{equation}
\label{eq:fe_u}
\nrm{u(\bsy) - u_h(\bsy)}{V} \,\leq\, C_{u} h,
\end{equation}
and for $\calG \in H^{-1 + t}(D)$ with $ t \in [0, 1]$
\begin{equation}
\label{eq:fe_G}
\big|\calG(u(\bsy)) - \calG(u_h(\bsy))\big| \,\leq\,C_{\calG} \,h^{1 + t},
\end{equation}
where $0 < C_{\lambda},\ C_{u},\ C_{\calG}$ are positive constants 
independent of $\bsy$ and $h$.
\end{theorem}

We have already seen that the minimal eigenvalue of the continuous problem
\eqref{eq:var-evp} is simple for all $\bsy$, and that the spectral gap is
bounded independently of $\bsy$.
It turns out that
the spectral gap of the FE eigenproblem \eqref{eq:fe-evp}
is also bounded independently of $\bsy$ and $h$,
provided that the FE eigenvalues are sufficiently accurate.
Specifically, if
\begin{equation}
\label{eq:hbar}
h \,\leq\, \hsuff \coloneqq \sqrt{\frac{\rho}{2C_{\lambda}}},
\end{equation} 
then
\begin{equation}
\label{eq:gap_h}
\lambda_{2, h}(\bsy) - \lambda_{1, h}(\bsy)
\,\geq\, \lambda_2(\bsy) - \lambda_1(\bsy) - \big(\lambda_{1, h}(\bsy) - \lambda_1(\bsy)\big)
\,\geq\, \rho - C_{\lambda} h^2
\,\geq\, \frac{\rho}{2},
\end{equation}
where we have used the FE error estimate \eqref{eq:fe_lam} and that $\lambda_{1, h}(\bsy)$ converges from above.

In fact, it is well known that for conforming methods all of the FE eigenvalues converge 
from above, so that $\lambda_{k, h}(\bsy) \geq \lambda_k(\bsy)$.
Then, as in \cite{GGKSS19}, we can 
use the eigenvalues of the Laplacian (or rather their FE approximations)
to bound the FE  eigenvalues and eigenfunctions independently of $\bsy$. 
Hence, for $k = 1, 2, \ldots, M_h$ and for all $\bsy \in \Omega$,
there exist $\overline{\lambda_{k}}$ and $\overline{u_{k}}$, 
which are independent of both $\bsy$ and $h$, such that
\begin{align}
\label{eq:lam_bnd}
\underline{\lambda_k} \,\coloneqq\, \frac{\amin}{\amax}\chi_k 
\,\leq\, \lambda_k(\bsy) \,\leq\, \lambda_{k, h}(\bsy) \,&\leq\,
\frac{\amax}{\amin} (\chi_{k, h} + 1)\,\leq\,\overline{\lambda_{k}},\\
\label{eq:u_bnd}
\max\big\{\nrm{u_k(\bsy)}{V},\ \nrm{u_{k, h}(\bsy)}{V} \big\}
\,&\leq\, \frac{\sqrt{\amax(\chi_{k, h} + 1)}}{\amin}
\,\leq\, \overline{u_{k}},
\end{align}
where $\chi_{k, h}$ is the FE approximation of the $k$th Laplacian eigenvalue $\chi_k$.
In addition to converging from above, 
for the Laplacian eigenvalues it is known that 
$\chi_k \leq \chi_{k, h} \leq \chi_k + C_kh^2$, for some constant that is 
independent of $h$ (see \cite[Theorem~10.4]{Boffi10}). As such, for $h$
sufficiently small there exists an upper bound on $\chi_{k, h}$ that is independent of $h$,
which in turn allows us to choose the final upper bounds $\overline{\lambda_{k}}$
and $\overline{u_{k}}$ so that they are independent of both $\bsy$ and $h$.

To conclude this section we introduce some notation and 
some properties of $V_h$ that will be useful later on.
First, the spaces $V_h$ satisfy the \emph{best approximation property}:
\begin{equation}
\label{eq:best_app}
\inf_{v_h \in V_h} \nrm{w - v_h}{V} \,\lesssim\, h \nrm{w}{Z},
\quad \text{for all } w \in Z.
\end{equation}

Then, for $h > 0$, let $P_h(\bsy) : V \to V_h$ denote the $\calA(\bsy)$-orthogonal 
projection of $V$ onto $V_h$, which satisfies
\begin{equation}
\label{eq:P_h}
\calA(\bsy;w - P_h(\bsy)w, v_h) \,=\, 0,
\quad \text{for all } w \in V,\  v_h \in V_h,
\end{equation}
and hence  also
\begin{equation*}
\nrm{w - P_h(\bsy)w}{\calA(\bsy)} \,=\, \inf_{v_h \in V_h} \nrm{w - v_h}{\calA(\bsy)}.
\end{equation*}

\subsection{Quasi-Monte Carlo integration}
\label{sec:qmc}
Quasi-Monte Carlo (QMC) methods are a class of equal-weight quadrature
rules that can be used to efficiently approximate an integral  over the $s$-dimensional (translated) unit cube
\[
\calI_sf \,\coloneqq\, \int_{[-\frac{1}{2}, \frac{1}{2}]^s} f(\bsy) \rd \bsy.
\]
There are several different flavours of QMC rules, however in this paper
we focus on \emph{randomly shifted rank-1 lattice rules}.
In Section~\ref{sec:hoqmc} we will also briefly discuss how to extend
our method to \emph{higher-order interlaced polynomial lattice rules},
see \cite{Dick08,GodaDick15}.
For further details on different QMC methods see, e.g., \cite{DKS13}.

A randomly shifted rank-1 lattice rule approximation to $\calI_sf$ 
using $N$ points is
\begin{equation}
\label{eq:rqmc}
Q_{s, N}(\bsDelta)f \,\coloneqq\, \frac{1}{N} 
\sum_{k = 0}^{N - 1} f(\bst_k - \tfrac{\boldsymbol{1}}{\boldsymbol{2}}),
\end{equation}
where for a \emph{generating vector} $\bsz \in \N^s$ 
and  a uniformly distributed \emph{random shift} $\bsDelta \in [0, 1)^s$,
the points $\bst_k$ are given by
\[
\bst_k \,=\, 
\bst_k(\bsDelta) \,=\, \bigg\{\frac{k\bsz}{N} + \bsDelta\bigg\} \quad \text{for } k = 0, 1, \ldots, N - 1.
\]
Here $\{\cdot\}$ denotes taking the fractional part of each component of a vector
and $\tfrac{\boldsymbol{1}}{\boldsymbol{2}} \coloneqq (\tfrac{1}{2}, \tfrac{1}{2}, \ldots \tfrac{1}{2})$.

The standard spaces for analysing randomly shifted lattices rules are the
so-called \emph{weighted} Sobolev spaces that were introduced in \cite{SW98}.
Here the term ``weighted'' is used to indicate that the space depends on a collection of 
positive numbers called ``weights'' 
that model the importance of different subsets of variables and enter the space
through its norm.
To be more explicit, given a collection of  weights 
$\bsgamma \coloneqq \{\gamma_\setu > 0 : \setu \subseteq \{1, 2, \ldots, s\}\}$,
let $\Ws$ be the $s$-dimensional weighted Sobolev space of functions with 
square-integrable mixed first derivatives, equipped with the (unanchored) norm
\begin{equation}
\label{eq:W-norm}
\nrm{f}{\Ws}^2
\,=\, \sum_{\setu \subseteq \{1:s\}} \frac{1}{\gamma_\setu} \int_{[-\frac{1}{2}, \frac{1}{2}]^{|\setu|}} 
\bigg(\int_{[-\frac{1}{2}, \frac{1}{2}]^{s - |\setu|}} \pd{|\setu|}{}{\bsy_\setu} f(\bsy) \, \rd \bsy_{-\setu}\bigg)^2
\rd \bsy_\setu.
\end{equation}
Here $\bsy_\setu \coloneqq (y_j)_{j \in \setu}$ and 
$\bsy_{-\setu} \coloneqq (y_j)_{j \in \{1:s\}\setminus \setu}$. 
Note also that we have used 
here \emph{set} notation to denote the mixed first derivatives,
as this is the convention in the QMC literature.
However, when we later give results for 
higher-order mixed derivatives we will switch to \emph{multi-index} notation.

A generating vector that leads to a good randomly shifted lattice rule in practice can
be constructed using the \emph{component-by-component} (CBC) algorithm, or the more efficient
\emph{fast CBC} construction \cite{NC06,NC06np}.
In particular, it can be shown (see, e.g., \cite[Theorem 5.10]{DKS13})
that the root-mean-square (RMS) error of a randomly shifted lattice rule 
using a generating vector constructed by the CBC algorithm satisfies
\begin{align}
\label{eq:cbc_err}
&\sqrt{\bbE_\bsDelta \big[ |\calI_s f - Q_{s, N}f|^2\big]}
\nonumber\\
&\qquad\leq\,
\Bigg(\frac{1}{\varphi(N)} \sum_{\emptyset \neq \setu \subseteq \{1:s\}}
\gamma_\setu^\xi \left( \frac{2\zeta(2\xi)}{(2\pi^2)^\xi}\right)^{|\setu|}
\Bigg)^{1/2\xi} \nrm{f}{\Ws}
\quad \text{for all } \xi \in (\tfrac{1}{2}, 1]. 
\end{align}
Here $\varphi$ is the Euler totient function, $\zeta$ is the Riemann zeta function
and $\bbE_\bsDelta$ denotes the expectation with respect to the random shift $\bsDelta$.
For $N$ prime one has $\varphi(N) = N -1$ or for $N$ a power of 2 one has $\varphi(N) = N/2$, 
and so in both cases taking  $\xi$ close to $1/2$
in \eqref{eq:cbc_err} results in the RMS error converging close to $\calO(N^{-1})$.

In practice, it is beneficial to perform several independent QMC approximations
corresponding to a small number of independent random shifts,
and then take the final approximation to be the average over the different shifts.
In particular, let $\bsDelta^{(1)}, \bsDelta^{(2)}, \ldots,  \bsDelta^{(R)}$ be $R$ independent 
uniform random shifts, and let the average over the QMC approximations
with random shift $\bsDelta^{(r)}$ be denoted by
\[
\Qhat_{s, N, R}f \,\coloneqq\, \frac{1}{R} \sum_{r = 1}^R Q_{s, N}(\bsDelta^{(r)})f .
\]
Then, the sample variance,
\begin{equation}
\label{eq:sample_var}
\Vhat[\Qhat_{s, N, R}] \,\coloneqq\, \frac{1}{R(R - 1)} \sum_{r = 1}^R \big[\Qhat_{s, N, R} f - Q_{s, N}(\bsDelta^{(r)})f\big]^2,
\end{equation}
can be used as an estimate of the mean-square error of $\Qhat_{s, N, R}f$.

\section{MLQMC for random EVPs}
\label{sec:mlqmc-alg}
Applying a QMC rule to each term in the telescoping sum \eqref{eq:tele_sum}, 
using a different number $N_\ell$ of samples on each level, 
a simple MLQMC approximation of $\bbE_\bsy[\lambda]$ is given by
\begin{equation}
\label{eq:mlqmc0}
\Qml_L(\bsDelta) \lambda
\,\coloneqq\, 
\sum_{\ell = 0}^L Q_\ell(\bsDelta_\ell)\big(\lambda_\ell - \lambda_{\ell - 1}\big).
\end{equation}
Here, we define $Q_\ell(\bsDelta_\ell)  \coloneqq Q_{s_\ell, N_\ell}(\bsDelta_\ell)$ (see \eqref{eq:rqmc}) and we treat the $L + 1$ independent random shifts, $\bsDelta_\ell \in [0, 1)^{s_\ell}$, 
as a single vector of dimension $\sum_{\ell = 0}^L s_\ell$, denoted by
$\bsDelta = (\bsDelta_0, \bsDelta_1, \ldots, \bsDelta_L)$. 
Recall also that  $\lambda_\ell = \lambda_{h_\ell, s_\ell}$
for $\ell = 0, 1, \ldots, L$, and for simplicity denote $\lambda_{-1} = 0$.
By using a different random shift for each level,
the approximations across different levels will be statistically independent.
For a linear functional $\calG \in V^*$, the MLQMC approximation to
$\bbE_\bsy[\calG(u)]$ is defined in a similar fashion.

As for single level QMC rules, it is beneficial to use multiple random shifts,
so that we can estimate the variance on each level. Letting 
$\bsDelta^{(1)}, \bsDelta^{(2)}, \ldots, \bsDelta^{(R)}$ be $R$ independent random shifts 
of dimension $\sum_{\ell = 0}^L s_\ell$, the shift-averaged MLQMC approximation is
\begin{equation}
\label{eq:mlqmc_R}
\Qhatml_{L, R}\lambda \,\coloneqq\, \sum_{\ell = 0}^L \frac{1}{R} \sum_{r = 1}^R
Q_\ell(\bsDelta_\ell^{(r)})\big(\lambda_\ell - \lambda_{\ell - 1}\big).
\end{equation}

If in practice the parameters are not specified beforehand, then we set 
$h_\ell \eqsim 2^{-\ell}$, $s_\ell \eqsim 2^\ell$ and use the adaptive algorithm 
from \cite{GilesWater09} to choose the number of QMC points $N_\ell$.

The mean-square error (with respect to the random shift(s) $\bsDelta$) 
of the MLQMC estimator can be written as the sum of the bias and 
the total variance as follows
\begin{equation}
\label{eq:err_decomp}
\bbE_\bsDelta\big[|\bbE_\bsy[\lambda] - \widehat{Q}_L(\bsDelta) \lambda |^2\big]
=\, |\bbE_\bsy[\lambda - \lambda_L]|^2
+ \sum_{\ell = 0}^L \bbV_\bsDelta[Q_\ell(\lambda_\ell - \lambda_{\ell - 1})].
\end{equation}
In the equation above, we have simplified the first term (corresponding to the bias) 
by the telescoping property, and the variance on each level is defined by
\begin{align*}
\bbV_\bsDelta[Q_\ell(\lambda_\ell - \lambda_{\ell - 1})] \,\coloneqq\, 
\bbE_\bsDelta\big[|\bbE_\bsy[\lambda_\ell - \lambda_{\ell - 1}] - Q_\ell(\bsDelta_\ell) (\lambda_\ell - \lambda_{\ell - 1})\big|^2\big],
\end{align*}
where the cross-terms have vanished because randomly shifted QMC rules are unbiased.
By the linearity of $\calG \in V^*$, the error for the eigenfunction approximation can be decomposed
in the same way.

Assuming that the total bias and the variance on each level decay at some given rates,
then the decomposition of the mean-square error \eqref{eq:err_decomp} gives the 
following abstract complexity theorems
(one each, for the eigenvalue and for functionals of the eigenfunction).
As is usual with the analysis of multilevel algorithms, the difficult part is to 
verify the assumptions on the decay of the variance and to determine the corresponding 
parameters. This analysis will be performed in Section~\ref{sec:err}.

\begin{theorem}[Eigenvalues]
\label{thm:ML_abstract_lam}
Suppose that
$\bbE_\bsDelta [Q_\ell(\lambda_\ell - \lambda_{\ell - 1})] = \bbE_\bsy[\lambda_\ell - \lambda_{\ell - 1}]$, and that there exist positive constants 
$\alpha_\lambda, \alpha', \beta_\lambda, \beta', \eta$
such that
\begin{enumerate}
\item[M1.] 
$|\bbE_\bsy [\lambda - \lambda_L]| \lesssim h_L^{\alpha_\lambda} + s_L^{-\alpha'}$,
and
\item[M2.] $\bbV_\bsDelta [Q_\ell(\lambda_\ell - \lambda_{\ell - 1})] 
\lesssim R^{-1}N_\ell^{-\eta} \Big(h_{\ell - 1}^{\beta_\lambda} + 
s_{\ell - 1}^{-\beta'}\Big)$,  for all $\ell = 0, 1, 2, \ldots, L$.
\end{enumerate}
Then
\begin{equation*}
\bbE_\bsDelta \Big[\big|\bbE_\bsy[\lambda] - \Qhatml_{L, R}(\lambda)\big|^2\Big]
\,\lesssim\,
h_L^{\alpha_\lambda} + s_L^{\alpha'} + \frac{1}{R}\sum_{\ell = 0}^L \frac{1}{N_\ell^{\eta}}
\Big(h_{\ell - 1}^{\beta_\lambda} + s_{\ell - 1}^{-\beta'}\Big).
\end{equation*}
\end{theorem}

\begin{theorem}[Functionals]
\label{thm:ML_abstract_G}
For $\calG \in V^*$, suppose
$\bbE_\bsDelta[\calG(u_\ell - u_{\ell - 1})] = \bbE_\bsy[\calG(u_\ell - u_{\ell - 1})]$,
and that there exist positive constants 
$ \alpha_\calG, \alpha', \beta_\calG,\beta', \eta$
such that
\begin{enumerate}
\item[M1.] $|\bbE_\bsy [\calG(u - u_L)]| \lesssim h_L^{\alpha_\calG} + s_L^{-\alpha'}$,
and
\item[M2.]
$\bbV_\bsDelta [Q_\ell(\calG(u_\ell - u_{\ell - 1}))] 
\lesssim R^{-1}N_\ell^{-\eta} \Big(h_{\ell - 1}^{\beta_\calG} + 
s_{\ell - 1}^{-\beta'}\Big)$, for all $\ell = 0, 1, 2, \ldots, L$.
\end{enumerate}
Then
\begin{equation*}
\bbE_\bsDelta \Big[\big|\bbE_\bsy[\calG(u)] - \Qhatml_{L, R}(\calG(u))\big|^2\Big] 
\,\lesssim\,
h_L^{\alpha_\calG} + s_L^{\alpha'} + \frac{1}{R}\sum_{\ell = 0}^L \frac{1}{N_\ell^{\eta}}
\Big(h_{\ell - 1}^{\beta_\calG} + s_{\ell - 1}^{-\beta'}\Big).
\end{equation*}
\end{theorem}

\begin{remark} In the case of a single truncation dimension, $s_\ell = s_L$ for all
$\ell = 1, 2, \ldots, L$, the terms $s_{\ell - 1}^{-\beta'}$ can be dropped from 
the theorems above.
\end{remark}

In Section~\ref{sec:err}, we verify that if Assumption~A\ref{asm:coeff} on the coefficients holds,
then Assumptions~M1 and M2 above are satisfied, and we give explicit values of the rates.

To better illustrate the power of our MLQMC algorithm, we
give here the following complexity bound for the special case of geometrically 
decaying meshwidths and a fixed truncation dimension.
We only give the eigenvalue result, but an analogous result holds also for 
linear functionals $\calG \in L^2(D)$. 
For less smooth functionals, $\calG \in H^{-1 + t}(D)$ for $t \in [0, 1]$,
similar results hold but with slightly adjusted rates.

\begin{corollary}
\label{cor:complexity}
Let $0 < \varepsilon \leq e^{-1}$
and suppose that Assumption~A\ref{asm:coeff} holds with $p, q \leq 2/3$. 
Also, let $h_\ell \eqsim 2^{-\ell}$ with $h_0$ sufficiently small and let
$s_\ell = s_L \eqsim h_L^{2p/(2 - p)}$. Finally, suppose that each $Q_\ell$ is an $N_\ell$-point 
lattice rule corresponding to a CBC-constructed generating vector.
If there exists $0 < \gamma < d + 1$ such that the cost on each level $\ell \in \N$ satisfies
\begin{enumerate}
\item[M3.] $\cost\big(Q_\ell(\lambda_\ell - \lambda_{\ell - 1})\big) \lesssim 
R N_\ell \big(s_\ell h_\ell^{-d} + h_\ell^{-\gamma}\big)$,
\end{enumerate}
then, $L$ and $N_\ell = 2^{n_\ell}$, for $n_\ell \in \N$, can be chosen such that
\begin{align*}
\bbE_\bsDelta \Big[\big|\bbE_\bsy[\lambda] - \Qhatml_{L, R}(\lambda)\big|^2\Big]
\,&\lesssim\, \varepsilon^2
\end{align*}
and for $\delta > 0$
\begin{equation*}
\cost\big(\Qhatml_{L, R}(\lambda)\big) \,\lesssim\,
\begin{cases}
\varepsilon^{-1 - p/(2 - p) - \delta} 
& \text{if } d = 1,\\
\varepsilon^{-1 - p/(2 - p) - \delta} \log_2(\varepsilon^{-1})^{3/2 + \delta}
&\text{if } d = 2,\\
\varepsilon^{-d/2 - p/(2 - p)} 
& \text{if } d > 2.
\end{cases}
\end{equation*}
\end{corollary}

\begin{proof}
In Section~\ref{sec:err} (cf., \eqref{eq:bias_lam} and Theorem~\ref{thm:Var_ell})
we verify that Assumptions M1, M2 from Theorem~\ref{thm:ML_abstract_lam} hold with
$\alpha_\lambda = 2$, $\alpha' = 2/p - 1$, $\beta_\lambda = 2\alpha_\lambda = 4$
and $\eta = 2 - \delta$.
The remainder of the proof follows by a standard minimisation argument 
as in, e.g., \cite[Cor.~2]{KSSSU17}.
\end{proof}

\begin{remark}
In \cite{GS21b} we verify that the cost does indeed satisfy Assumption M3 with 
$\gamma \approx d$, which is the same order cost as the source problem.
\end{remark}

\section{Stochastic regularity}
\label{sec:reg}
In order for a randomly shifted lattice rule approximation to achieve the error bound
\eqref{eq:cbc_err}, we require that the integrand belongs to $\Ws$, 
which in turn requires bounds on the mixed first derivatives.
For the eigenproblem \eqref{eq:evp}, this means that we need
to study the regularity of eigenvalues (and eigenfunctions)
with respect to the stochastic parameter $\bsy$.
In order to bound the variance on each level of our MLQMC estimator,
it is necessary to also study the FE error
in $\Ws$ (cf. \eqref{eq:err_ell_decomp}), whereas the single level analysis 
in \cite{GGKSS19} only required the expected FE error.
This analysis of the FE error in a stronger norm requires mixed regularity of the 
solution with respect to both $\bsx$ and $\bsy$ simultaneously, 
which has not been shown previously.
The theorem below presents the required bounds for $u$,
along with the bounds from \cite{GGKSS19} with 
respect to $\bsy$ only, which are included here for completeness.
Analyticity of simple eigenvalues and eigenfunctions with respect to $\bsy$ was 
shown in \cite{AS12}, however, explicit bounds on the derivatives were not given there 
and they also did not consider the mixed $\bsx$ and $\bsy$ regularity required for 
the ML analysis.

Although the analysis of randomly shifted lattice rules requires only the mixed \emph{first}
derivatives (cf., \eqref{eq:W-norm}), we also give results for arbitrary higher-order mixed
derivatives. We do this because the proof technique is the same, and also since these bounds may 
be useful for the analysis of higher-order methods, e.g., higher-order QMC 
(see Section~\ref{sec:hoqmc}) or sparse grid rules (see, e.g., \cite{GrieHarMul20,ZechDungSchw19}).
As such, to simplify notation we will write mixed higher-order derivatives using 
multi-index notation instead of the set notation used in Section~\ref{sec:qmc}. 
For a multi-index $\bsnu = (\nu_j)_{j \in \N}$ with $\nu_j \in \N \cup \{0 \}$
and only finitely-many nonzero components,
let $\pdy^\bsnu$ denote the mixed partial differential operator where the order of derivative
with respect to the variable $y_j$ is $\nu_j$.
Define $|\bsnu| \coloneqq \sum_{j \geq 1} \nu_j$ and denote the
set of all admissible multi-indices by $\indx \coloneqq \{ \bsnu \in \N^\N: |\bsnu| < \infty\}$.
All operations and relations between multi-indices will be performed componentwise,
e.g., for $\bsnu, \bsm \in \calF$ addition is given by 
$\bsnu + \bsm = (\nu_j + m_j)_{j\in \N}$, and
$\bsnu \leq \bsm$ if and only if $\nu_j \leq m_j$ for all $j \in \N$.
Similarly, 
for $\bsnu, \bsm \in \calF$ and a sequence $\bsbeta \in \ell^\infty$
define the following shorthand for products
\[
\binom{\bsnu}{\bsm} \,\coloneqq\, \prod_{j = 1}^\infty \binom{\nu_j}{m_j}
\quad \text{and} \quad
\bsbeta^\bsnu \,\coloneqq\, \prod_{j = 1}^\infty \beta_j^{\nu_j}.
\]
Note that since $\bsnu, \bsm \in \calF$ have finite support
these products have finitely-many terms.

\begin{theorem}\label{thm:regularity_y}
Let $\bsnu \in \indx$ be a multi-index, let $\upepsilon \in (0, 1)$, and suppose that 
Assumption~A\ref{asm:coeff} holds.
Also, define the sequences $\bsbeta = (\beta_j)_{j \in \N}$ and 
$\bsbetabar = (\betabar_j)_{j \in \N}$ by
\begin{align}
\label{eq:beta}
\beta_j \,&\coloneqq\, C_\bsbeta \max\big(\nrm{a_j}{L^\infty}, \nrm{b_j}{L^\infty}\big),\\
\label{eq:betabar}
\betabar_j \,&\coloneqq\, C_\bsbeta 
\max\big(\nrm{a_j}{L^\infty}, \nrm{b_j}{L^\infty}, \nrm{\grad a_j}{L^\infty}\big),
\end{align}
where $C_\bsbeta \geq 1$,
given explicitly below in \eqref{eq:C_beta},
is independent of $\bsy$ but depends on $\upepsilon$.

Then, for all $\bsy \in \Omega$, the derivative of the minimal eigenvalue with respect to 
$\bsy$ is bounded by
\begin{equation}
\label{eq:dlambda}
|\pdy^\bsnu \lambda(\bsy)| \,\leq\, 
\overline{\lambda} \, |\bsnu|!^{1 + \upepsilon} \, \bsbeta^\bsnu,
\end{equation}
and the derivative of the corresponding eigenfunction satisfies both
\begin{align}
\label{eq:du_V}
\nrm{\pdy^\bsnu u(\bsy)}{V} 
\,&\leq\, \overline{u} \, |\bsnu|!^{1 + \upepsilon} \,\bsbeta^\bsnu,\\
\label{eq:du_Lap}
\nrm{\pdy^\bsnu u(\bsy)}{Z} \,&\leq\, C\,|\bsnu|!^{1 + \upepsilon} \, \bsbetabar^\bsnu,
\end{align}
where $\overline{\lambda}$, $\overline{u}$ are as in \eqref{eq:lam_bnd}, \eqref{eq:u_bnd},
respectively, and $C$ in \eqref{eq:du_Lap} is independent of 
$\bsy$ but depends on $\upepsilon$.

Moreover, for $h > 0$ sufficiently small, the bounds \eqref{eq:dlambda} and \eqref{eq:du_V}
are also satisfied by $\lambda_h(\bsy)$ and $u_h(\bsy)$,
respectively.

\end{theorem}

\begin{proof}
To facilitate the proof with a single constant for both sequences 
$\bsbeta$ and $\bsbetabar$ we define
\begin{equation}
\label{eq:C_beta}
C_\bsbeta \,\coloneqq\, \frac{2\overline{\lambda_2}}{\rho} 
\frac{\amin\overline{\lambda}}{\amax^2\underline{\lambda}}
\bigg(\frac{3\overline{\lambda}}{\underline{\lambda}} C_\upepsilon + 1\bigg),
\end{equation}
where $C_\upepsilon$ from \cite[Lemma~3.3]{GGKSS19} is given by
\[
C_\upepsilon \,\coloneqq\,
\frac{2^{1 - \upepsilon}}{1 - 2^{- \upepsilon}}
\bigg(\frac{e^2}{\sqrt{2\pi}}\bigg)^\upepsilon,
\]
which is independent of $\bsy$ and $h$.
Then clearly it follows that $C_\bsbeta$ is independent of $\bsy$ and $\bsh$.
Later we will use that $1/\amin \leq C_\bsbeta/2$ and $1/(\amin\chi^{1/2}) \leq C_\bsbeta/2$,
which both follow from the lower bounds $C_{\upepsilon} \geq 1$ for all $\upepsilon \in (0, 1)$ and $\overline{\lambda}/\underline{\lambda} \geq \amax^{2}/\amin^{2}(1 + 1/\chi)$.

The proof for the bounds \eqref{eq:dlambda} and \eqref{eq:du_V} is given
in \cite[Theorem~3.4]{GGKSS19}. 
If $h$ is sufficiently small such that the FE eigenvalues resolve 
the spectral gap (i.e., \eqref{eq:gap_h} holds)
then the bounds also hold for $\lambda_h(\bsy)$ and
$u_h(\bsy)$ because $V_h \subset V$, cf. \cite[Rem.~3.2 and 3.5]{GGKSS19}.

For the bound \eqref{eq:du_Lap}, we first prove a recursive bound
on $\|\partial^\bsnu u(\bsy)\|_Z$ and then use an induction result
from \cite{DKLeGS16} to prove the final bound.
Consider the strong form of the eigenproblem
\eqref{eq:evp} for the pair $(\lambda(\bsy), u(\bsy))$, which,
omitting the $\bsx$ and $\bsy$ dependence, is given by
\[
-\nabla \cdot (a \nabla u) + b u \,=\, c\lambda u.
\]

The $\bsnu$th derivative with respect to $\bsy$ commutes with the spatial derivatives $\nabla$.
Thus, using the Leibniz general product rule we have
\begin{align*}
-\nabla \cdot (a \nabla \pdy^\bsnu u) + b \pdy^\bsnu u 
+ \sum_{j = 1}^\infty \nu_j \big(-\nabla\cdot (a_j &\nabla \pdy^{\bsnu - \bse_j} u)
- b_j \pdy^{\bsnu - \bse_j} u\big)\\
\,&=\,
c \sum_{\bsm \leq \bsnu} \binom{\bsnu}{\bsm} \pdy^\bsm \lambda \ \pdy^{\bsnu - \bsm} u,
\end{align*}
where $\bse_j$ is the multi-index that is 1 in the $j$th entry and
zero elsewhere.
Then we can use the identity $\nabla \cdot (\phi \bspsi) = \phi \nabla \cdot \bspsi + \nabla \phi \cdot \bspsi$ to simplify this to
\begin{align*}
a\Delta \pdy^\bsnu u \,=\,
&- \grad a \cdot \grad \pdy^{\bsnu}u + b \pdy^\bsnu u
- c\sum_{\bsm \leq \bsnu} \binom{\bsnu}{\bsm} \pdy^\bsm \lambda \ 
\pdy^{\bsnu - \bsm} u\\
&+ \sum_{j = 1}^\infty \nu_j \big(-a_j \Delta \pdy^{\bsnu - \bse_j} u - \grad a_j \cdot \grad \pdy^{\bsnu - \bse}u
+ b_j \pdy^{\bsnu - \bse_j} u\big) .
\end{align*}

Since 
$a \geq \amin > 0;$ 
$a, a_j \in W^{1, \infty}$ and $b, b_j \in L^\infty$ for all $ j \in \N$;
and $\pdy^\bsm u \in V$ for all $\bsm \in \indx$,
it follows by induction on $|\bsnu|$ that 
$\Lap \pdy^\bsnu u \in L^2$.
This allows us to take the $L^2$-norm of both sides, which, after using the triangle inequality
and the bounds in \eqref{eq:coeff_bnd}, gives the following recursive bound
for $\Lap \pdy^\bsnu u$
\begin{align*}
\nrm{\Lap \pdy^\bsnu u}{L^2} \,\leq\, &
\frac{\nrm{\grad a}{L^\infty}}{\amin}\nrm{\pdy^\bsnu u}{V} + 
\frac{\nrm{b}{L^\infty}}{\amin} \nrm{\pdy^\bsnu u}{L^2}
\\
&+ \frac{\nrm{c}{L^\infty}}{\amin} \sum_{\bsm \leq \bsnu} \binom{\bsnu}{\bsm} 
|\pdy^\bsm \lambda | \nrm{\pdy^{\bsnu - \bsm} u}{L^2} 
+ \sum_{j = 1}^\infty \nu_j \frac{\nrm{a_j}{L^\infty}}{\amin} \nrm{\Lap \pdy^{\bsnu - \bse_j} u}{L^2}
\\
&+ \frac{1}{\amin}\sum_{j = 1}^\infty \nu_j \big(
\nrm{\grad a_j}{L^\infty} \nrm{\pdy^{\bsnu - \bse_j} u}{V}
+ \nrm{b_j}{L^\infty} \nrm{\pdy^{\bsnu - \bse_j}u}{L^2}\big) .
\end{align*}

Adding $\nrm{\pdy^\bsnu u}{L^2}$ to both sides and then using the definition of $\betabar_j$, 
we can write this bound in terms of the $Z$-norm as
\begin{equation}
\label{eq:du_Z_rec}
\nrm{\pdy^\bsnu u}{Z} \,\leq\, 
\nrm{\pdy^\bsnu u}{L^2} + \nrm{\Lap \pdy^\bsnu u}{L^2} \,\leq\,
\sum_{j = 1}^\infty \nu_j \betabar_j \nrm{\pdy^{\bsnu - \bse_j} u}{Z} + B_\bsnu,
\end{equation}
where we used that $1/\amin \leq C_\bsbeta$, and then defined
\begin{align*}
&B_\bsnu \,\coloneqq\, 
\frac{\nrm{\grad a}{L^\infty}}{\amin}\nrm{\pdy^\bsnu u}{V}
+
\frac{\nrm{c}{L^\infty}}{\amin} \sum_{\bsm \leq \bsnu} \binom{\bsnu}{\bsm} 
|\pdy^\bsm \lambda | \nrm{\pdy^{\bsnu - \bsm} u}{L^2}
\\
&\bigg(\frac{\nrm{b}{L^\infty}}{\amin} + 1\bigg) \nrm{\pdy^\bsnu u}{L^2}
+ \frac{1}{\amin}\sum_{j = 1}^\infty \nu_j \big(
\nrm{\grad a_j}{L^\infty} \nrm{\pdy^{\bsnu - \bse_j} u}{V}
+ \nrm{b_j}{L^\infty} \nrm{\pdy^{\bsnu - \bse_j}u}{L^2}\big).
\end{align*}

Now, the sum on the right of \eqref{eq:du_Z_rec} only involves lower-order versions 
of the object we are interested in bounding (namely, $\nrm{\pdy^\bsnu u}{Z}$), whereas
the terms in $B_\bsnu$ only involve derivatives that can be
bounded using one of \eqref{eq:dlambda} or \eqref{eq:du_V}.

We bound the remaining $L^2$-norms in $B_\bsnu$ by the Poincar\'e
inequality \eqref{eq:poin} to give
\begin{align*}
B_\bsnu \,\leq\, 
& \bigg(\frac{\nrm{\grad a}{L^\infty}}{\amin}+ 
\frac{\nrm{b}{L^\infty} + \amin}{\amin\sqrt{\evalueLap}}\bigg) \nrm{\pdy^\bsnu u}{V}
+ \frac{\nrm{c}{L^\infty}}{\amin\sqrt{\evalueLap}} 
\sum_{\bsm \leq \bsnu} \binom{\bsnu}{\bsm} 
|\pdy^\bsm \lambda | \nrm{\pdy^{\bsnu - \bsm} u}{V} 
\\
& + \frac{1}{\amin}\sum_{j = 1}^\infty \nu_j \bigg(
\nrm{\grad a_j}{L^\infty}
+ \frac{\nrm{b_j}{L^\infty}}{\sqrt{\evalueLap}} \bigg)
\nrm{\pdy^{\bsnu - \bse_j}u}{V}
\\
\leq\, 
&\frac{\amax}{\amin} \bigg(1 +  \frac{2}{\sqrt{\evalueLap}}\bigg)
\Bigg( \nrm{\pdy^\bsnu u}{V} + \sum_{\bsm \leq \bsnu} \binom{\bsnu}{\bsm} 
|\pdy^\bsm \lambda | \nrm{\pdy^{\bsnu - \bsm} u}{V} \Bigg)
\\
&+\frac{1}{\amin}\sum_{j = 1}^\infty \nu_j \bigg(
\nrm{\grad a_j}{L^\infty}
+ \frac{\nrm{b_j}{L^\infty}}{\sqrt{\evalueLap}} \bigg)
\nrm{\pdy^{\bsnu - \bse_j}u}{V},
\end{align*}
where in the last inequality we have bounded the  $L^\infty$-norms on the second line
using \eqref{eq:coeff_bnd}, and then simplified.
Then, substituting in the bounds \eqref{eq:dlambda} and \eqref{eq:du_V} gives
\begin{align*}
B_\bsnu \,\leq\, 
& \frac{\amax}{\amin} \bigg(1 +  \frac{2}{\sqrt{\evalueLap}}\bigg)
\Bigg(\ubar |\bsnu|!^{1 + \upepsilon} \bsbeta^\bsnu
+ \sum_{\bsm \leq \bsnu} \binom{\bsnu}{\bsm} 
\lambdabar |\bsm|!^{1 + \upepsilon} \bsbeta^\bsm \cdot
\ubar |\bsnu - \bsm|!^{1 + \upepsilon} \bsbeta^{\bsnu - \bsm}\Bigg)
\\
&+\frac{1}{\amin}\sum_{j = 1}^\infty \nu_j \bigg(
\nrm{\grad a_j}{L^\infty} + \frac{\nrm{b_j}{L^\infty}}{\sqrt{\evalueLap}} \bigg)
\ubar (|\bsnu| - 1)!^{1 + \upepsilon} \bsbeta^{\bsnu - \bse_j}
\\
\,=\, 
&\ubar\frac{\amax}{\amin} \bigg(1 +  \frac{2}{\sqrt{\evalueLap}}\bigg) \bsbeta^\bsnu
\Bigg( |\bsnu|!^{1 + \upepsilon} + 
\lambdabar\sum_{\bsm \leq \bsnu} \binom{\bsnu}{\bsm} 
|\bsm|!^{1 + \upepsilon} |\bsnu - \bsm|!^{1 + \upepsilon} \Bigg)\\
&+ \frac{\ubar}{\amin} (|\bsnu| - 1)!^{1 + \upepsilon}
\sum_{j = 1}^\infty \nu_j \bigg(
\nrm{\grad a_j}{L^\infty} + \frac{\nrm{b_j}{L^\infty}}{\sqrt{\evalueLap}} \bigg)
\bsbeta^{\bsnu - \bse_j} .
\end{align*}

Using the fact that $(1 + \chi^{-1/2})/\amin \leq C_\bsbeta$
and also that clearly $\beta_j \leq \betabar_j$, we have
\begin{align*}
B_\bsnu \,\leq\, &
\ubar\frac{\amax}{\amin} \bigg(1 +  \frac{2}{\sqrt{\evalueLap}}\bigg) \bsbetabar^\bsnu
\Bigg(|\bsnu|!^{1 + \upepsilon} + 
\lambdabar\sum_{\bsm \leq \bsnu} \binom{\bsnu}{\bsm} 
|\bsm|!^{1 + \upepsilon} |\bsnu - \bsm|!^{1 + \upepsilon} \Bigg)\\
& + \frac{\ubar}{\amin} (|\bsnu| - 1)!^{1 + \upepsilon} \sum_{j = 1}^\infty \nu_j 
\big(1 + \chi^{-1/2}\big) \max\big(\nrm{\grad a_j}{L^\infty}, \nrm{b_j}{L^\infty}\big) 
\bsbetabar^{\bsnu - \bse_j}
\\
\leq\,&
\ubar\frac{\amax}{\amin} \bigg(1 +  \frac{2}{\sqrt{\evalueLap}}\bigg) \bsbetabar^\bsnu
\Bigg(|\bsnu|!^{1 + \upepsilon} + \lambdabar\sum_{\bsm \leq \bsnu} \binom{\bsnu}{\bsm} 
|\bsm|!^{1 + \upepsilon} |\bsnu - \bsm|!^{1 + \upepsilon} \Bigg)
\\
&+ \ubar |\bsnu|!^{1 + \upepsilon}\,
\bsbetabar^{\bsnu}.
\end{align*}

The sum that remains can be bounded using the same strategy as in the proof of
\cite[Lemma~3.4]{GGKSS19}, as follows
\begin{align}
\label{eq:sum_bnd}
\nonumber
\sum_{\bsm \leq \bsnu} \binom{\bsnu}{\bsm} |\bsm|!^{1 + \upepsilon} |\bsnu - \bsm|!^{1 + \upepsilon}
=\, & 2|\bsnu|!^{1 + \upepsilon} + \sum_{k = 1}^{|\bsnu| - 1}  k!^{1 + \upepsilon} (|\bsnu| - k)!^{1 + \upepsilon}
\sum_{\bsm \leq \bsnu, |\bsm| = k} \binom{\bsnu}{\bsm}
\nonumber\\
=\, & 2|\bsnu|!^{1 + \upepsilon} + \sum_{k = 1}^{|\bsnu| - 1}  k!^{1 + \upepsilon} (|\bsnu| - k)!^{1 + \upepsilon}
\binom{|\bsnu|}{k}
\nonumber\\
=\, &|\bsnu|!^{1 + \upepsilon}\bigg(2 +  \sum_{k = 1}^{|\bsnu| - 1} \binom{|\bsnu|}{k}^{-\upepsilon}\bigg)
\nonumber\\
\leq\, & |\bsnu|!^{1 + \upepsilon} \Bigg(2 + 
\underbrace{\frac{2^{1 - \upepsilon}}{1 - 2^{-\upepsilon}} \bigg(\frac{e^2}{\sqrt{2\pi}}\bigg)^\upepsilon}_{C_\upepsilon}\Bigg)
,
\end{align}
where for the inequality on the last line we have used \cite[Lemma~3.3]{GGKSS19}.

Hence, $B_\bsnu$ is bounded above by
\[
B_\bsnu \,\leq\, C_B |\bsnu|!^{1 + \upepsilon}\, \bsbetabar^\bsnu,
\]
where
\[
C_B \,\coloneqq\, \ubar \bigg[\frac{\amax}{\amin} \bigg(1 + \frac{2}{\sqrt{\evalueLap}}\bigg)
\big(1 + \lambdabar(2 + C_\upepsilon)\big) + 1 \bigg]
\,<\, \infty
\]
is clearly independent of $\bsy$ and $\bsnu$.

Now we can bound the recursive formula \eqref{eq:du_Z_rec} using 
the bound above on $B_\bsnu$, which gives
\[
\nrm{\pdy^\bsnu u}{Z} \,\leq\, 
\sum_{j = 1}^\infty \nu_j \betabar_j \nrm{\pdy^{\bsnu - \bse_j} u}{Z} + C_B |\bsnu|!^{1 + \upepsilon} \,\bsbetabar^\bsnu.
\]

Finally, by \cite[Lemma~4]{DKLeGS16} we can bound this above by
\begin{align*}
\nrm{\pdy^\bsnu u}{Z} \,&\leq\, \sum_{\bsm \leq\, \bsnu} \binom{\bsnu}{\bsm}
|\bsm|! \bsbetabar^\bsm \ C_B |\bsnu - \bsm|!^{1 + \upepsilon}\, \bsbetabar^{\bsnu - \bsm}
\\
\,&=\, C_B \bsbetabar^\bsnu 
\sum_{\bsm \leq\, \bsnu} \binom{\bsnu}{\bsm}
|\bsm|! |\bsnu - \bsm|!^{1 + \upepsilon}
\,\leq\, C_B (2 + C_\upepsilon) |\bsnu|!^{1 + \upepsilon}\, \bsbetabar^\bsnu,
\end{align*}
where to obtain the final result we have again used \eqref{eq:sum_bnd}.
\end{proof}

\section{Error analysis}
\label{sec:err}
We now provide a rigorous analysis of the error for \eqref{eq:mlqmc0}, which we do by
verifying the assumptions from Theorems~\ref{thm:ML_abstract_lam} and 
\ref{thm:ML_abstract_G}.

Recall that we use the shorthand $\lambda_\ell \coloneqq \lambda_{h_\ell, s_\ell}$ 
for the dimension-truncated FE approximation of the minimal eigenvalue on level 
$\ell$, whereas
$\lambda_s$ denotes the minimal eigenvalue of the dimension-truncated version
of the continuous EVP \eqref{eq:trunc-evp}.
The bias (the first term) in \eqref{eq:err_decomp} can be bounded by the triangle inequality to give
\[
|\bbE_\bsy[\lambda - \lambda_{h_L, s_L}] |
\,\leq\, |\bbE_\bsy[\lambda - \lambda_{s_L}] | + |\bbE_\bsy[\lambda_{s_L} - \lambda_{h_L, s_L}]|,
\]
and similarly for the eigenfunction.
Now, both terms on the right can be bounded above using the results from the single level
algorithm. Explicitly, for $\calG \in H^{-1 + t}(D)$ with $t \in [0, 1]$ using 
Theorem~4.1  from \cite{GGKSS19} and then Theorem~\ref{thm:fe_err} gives the bounds
\begin{align}
\label{eq:bias_lam}
|\bbE_\bsy[\lambda - \lambda_L] | \,&\lesssim\, s_L^{-2/p + 1} + h_L^2,\\
\label{eq:bias_u}
|\bbE_\bsy[\calG(u - u_L)]| \,&\lesssim\, s_L^{-2/p + 1} + h_L^{1 + t},
\end{align}
with constants independent of $s_L$ and $h_L$.
That is, we have verified Assumptions~M1 from both Theorems~\ref{thm:ML_abstract_lam} and 
\ref{thm:ML_abstract_G} 
with $\alpha_\lambda = 2$, $\alpha_\calG = 1 + t$ and  $\alpha' = 2/p - 1$.

For the variance terms on each level in \eqref{eq:err_decomp} 
(alternatively to verify Assumption~M2), 
we must study the QMC error of the differences 
$\lambda_\ell - \lambda_{\ell - 1}$. 
Since $\lambda_\ell - \lambda_{\ell - 1} \in \calW_{s_\ell, \bsgamma}$ for all $\ell = 0, 1, 2, \ldots, L$ and each QMC rule $Q_\ell$ uses CBC-constructed generating vector $\bsz_\ell$, by \eqref{eq:cbc_err} we have the upper bound
\begin{equation}
\label{eq:Var_ell}
\bbV_\bsDelta[Q_\ell(\lambda_\ell - \lambda_{\ell - 1})] 
\,\leq\, \frac{C_{\xi, \ell}^2}{\varphi(N_\ell)^{1/\xi}} 
\nrm{\lambda_\ell - \lambda_{\ell - 1}}{\calW_{s_\ell, \bsgamma}}^2
\quad \text{for all } \xi \in (1/2, 1],
\end{equation}
where $C_{\xi, \ell}$ is the constant from \eqref{eq:cbc_err} with $s = s_\ell$.
Thus, in Assumption~M2 we can take $\eta = 1/\xi \in [1, 2)$ and for the other parameters
we must study the norm of the difference on each level.

By the triangle inequality, we can separate truncation and FE components of
the error
\begin{align}
\label{eq:err_ell_decomp}
\|\lambda_\ell& - \lambda_{\ell - 1}\|_{\calW_{s_\ell, \bsgamma}}
\nonumber\\
\,&\leq\, 
\nrm{\lambda_{s_\ell} - \lambda_{s_{\ell - 1}}}{\calW_{s_\ell, \bsgamma}}
+ \nrm{\lambda_{s_\ell} - \lambda_{h_\ell, s_\ell}}{\calW_{s_\ell, \bsgamma}}
+ \nrm{\lambda_{s_{\ell - 1}} - \lambda_{h_{\ell - 1}, s_{\ell - 1}}}{\calW_{s_{\ell - 1}, \bsgamma}}.
\end{align}
In contrast to the single level setting \cite{GGKSS19}, here we need to study
the truncation and FE errors in the weighted QMC norm \eqref{eq:W-norm}
instead of simply the expected truncation and FE errors.
Each term will be handled separately in the subsections that follow.

The key ingredient in the error analysis are the bounds of the
derivatives of the minimal eigenvalue and its eigenfunction that were given in 
Section~\ref{sec:reg}.

\subsection{Estimating the FE error}
\label{sec:fe-err}
As a first step towards bounding the FE errors in the $\Ws$-norm, we bound
their derivatives with respect to $\bsy$, which are given below in Theorem~\ref{thm:du-fe}.
The bulk of the work to bound the FE error in $\Ws$ is dedicated to 
proving these regularity bounds.
As in Theorem~\ref{thm:regularity_y} we also present bounds on higher-order 
mixed derivatives instead of simply the mixed first derivatives required in the
$\Ws$ norm.

The strategy for proving these bounds is similar to the proof 
\cite[Lemma~3.4]{GGKSS19}, except in the current multilevel setting
we need to bound the derivatives of the FE errors of the eigenvalue and eigenfunction,
in addition to the derivatives of the eigenvalue and eigenfunction themselves.
First, we differentiate variational
equations involving the errors to obtain a recursive formula for each of 
the eigenvalue and eigenfunction errors,  and then prove the bounds by induction
on the cardinality of $|\bsnu|$. 
Once we have proved the bound for the eigenfunction in \eqref{eq:du-fe},
the result for any functional $\calG(u(\bsy))$ in \eqref{eq:dGu-fe} 
follows by a duality argument.
Throughout the proofs in this section we will omit
the $\bsx$ and $\bsy$ dependence. Note also that throughout we must explicitly 
track the constants to ensure that they are independent of $\bsy$ and $h$, 
but also to make sure that in both of the inductive steps the constants are not growing, 
since this could interfere with the summability of $\bsbetahat$.
Also, the results in this section are all shown for $h$ sufficiently small, where here
\emph{sufficiently small} means that the FE eigenvalues resolve the spectral 
gap. Explicitly, we assume that $h \leq \hsuff$ (see \eqref{eq:hbar} and \eqref{eq:gap_h})
for some $\hsuff > 0 $ that is independent of $\bsy$.
This ensures that the condition that $h$
is sufficiently small (i.e., $h \leq \hsuff$) is also independent of $\bsy$.

In the following key lemma, we bound the derivative of the difference between the eigenfunction and its projection $\Ph u(\bsy)$ onto $V_h$, which is not equal to the 
FE eigenfunction $u_h(\bsy)$, but is easier to handle. 
The proof relies on the new mixed regularity estimate \eqref{eq:du_Lap}.

\begin{lemma}
Let $\bsnu \in \indx$ be a multi-index, let $h > 0$ be sufficiently small and suppose that Assumption~A\ref{asm:coeff} holds. Then
\begin{equation}
\label{eq:du_proj_err}
\nrm{\pdy^\bsnu u - \Ph \pdy^\bsnu u}{V}
\,\leq\,  C_{\calP} \, h |\bsnu|!^{1 + \upepsilon} \bsbetabar^\bsnu,
\end{equation}
where $\bsbetabar$ is as defined in \eqref{eq:betabar}, and
$C_{\calP}$ is independent of $\bsy$, $h$ and $\bsnu$.
\end{lemma}
\begin{proof}
Using the equivalence of the $V$-norm and the induced $\calA$-norm in \eqref{eq:A_equiv}, 
along with the $\calA$-orthogonality of the projection and the best approximation 
property \eqref{eq:best_app}, we get
\begin{align*}
\nrm{\pdy^\bsnu u - \Ph \pdy^\bsnu u}{V} 
\,&\leq\, \sqrt{\frac{\amax}{\amin}\bigg(1 + \frac{1}{\evalueLap}\bigg)} \inf_{v_h \in V_h} \nrm{\pdy^\bsnu u - v_h}{V}
\\
\,&\leq\,  \sqrt{\frac{\amax}{\amin}\bigg(1 + \frac{1}{\evalueLap}\bigg)} 
C h \|\pdy^\bsnu u\|_Z
\,\leq\, C_{\calP} h |\bsnu |!^{1 + \upepsilon} \bsbetabar^\bsnu,
\end{align*}
where for the last inequality we have used the bound \eqref{eq:du_Lap}.
The final constant $C_{\calP}$ is  independent of $\bsy$, $h$ and also $\bsnu$.
\end{proof}

The three recursive formulae presented in the next two lemmas are the key to the
induction proof to bound the derivatives of the FE error.
The general strategy is to differentiate variational equations involving the FE errors.
However, the proofs are quite long and technical, and as such are deferred to 
the Appendix.

\begin{lemma}
\label{lem:dlam_fe_rec}
Let $\bsnu \in \indx$ be a multi-index, let $h > 0$ be sufficiently small and suppose that Assumption~A\ref{asm:coeff} holds.
Then, for all $\bsy \in \Omega$, the following two recursive bounds hold
\begin{align}
\label{eq:dlam_fe_rec_a}
&|\pdy^\bsnu(\lambda - \lambda_h)| \,\leq\,
\CI\Bigg( h |\bsnu|! \bsbetabar^\bsnu + 
\sum_{j = 1}^\infty \nu_j \beta_j
\nrm{\pdy^{\bsnu - \bse_j}(u - u_h)}{V}
\nonumber\\ &
+\sum_{\substack{\bszero \neq \bsm \leq \bsnu \\ \bsm \neq \bsnu}} 
\binom{\bsnu}{\bsm} |\bsm|!^{1 + \upepsilon} \bsbeta^{\bsm}
\Big[\nrm{\pdy^{\bsnu - \bsm}(u - u_h)}{V}
+ |\pdy^{\bsnu - \bsm}(\lambda - \lambda_h)|\Big] \Bigg)
\end{align}
and
\begin{align}\label{eq:dlam_fe_rec_b}
&|\pdy^\bsnu(\lambda - \lambda_h)| \,\leq\,
\CII\Bigg(\sum_{\bsm \leq\bsnu} \binom{\bsnu}{\bsm}
\nrm{\pdy^{\bsnu - \bsm}(u - u_h)}{V} \nrm{\pdy^\bsm(u - u_h)}{V}
\\\nonumber
&+ \sum_{j = 1}^\infty\sum_{\bsm \leq \bsnu - \bse_j} 
\nu_j \binom{\bsnu - \bse_j}{\bsm} \beta_j 
\nrm{ \pdy^{\bsnu - \bse_j - \bsm}(u - u_h)}{V}
\nrm{\pdy^\bsm (u - u_h)}{V}
\\\nonumber
&+ \sum_{\bsm \leq \bsnu} \sum_{\bsk \leq \bsm}
\binom{\bsnu}{\bsm} \binom{\bsm}{\bsk} 
|\bsnu - \bsm|!^{1 + \upepsilon} \bsbeta^{\bsnu - \bsm}
\nrm{\pdy^{\bsm - \bsk}(u - u_h)}{V} \nrm{\pdy^\bsk(u - u_h))}{V}\Bigg),
\end{align}
where $\bsbeta$, $\bsbetabar$ are defined in \eqref{eq:beta}, \eqref{eq:betabar}, respectively,
and $\CI,\ \CII$ are independent of $\bsy$, $h$ and $\bsnu$.
\end{lemma}

\begin{lemma}
\label{lem:du-fe_rec}
Let $\bsnu \in \indx$ be a multi-index, let $h > 0$ be sufficiently small and suppose that Assumption~A\ref{asm:coeff} holds.
Then, for all $\bsy \in \Omega$,
\begin{align}
\label{eq:du_fe_rec}
\nonumber\nrm{\pdy^\bsnu (u -& u_h)}{V} 
\leq\, 
\CIII \Bigg( h |\bsnu|!^{1 + \upepsilon} \bsbetabar^\bsnu
+
\sum_{j = 1}^\infty \nu_j \beta_j
\nrm{\pdy^{\bsnu - \bse_j} (u - u_h)}{V}
\\
&+\,
\sum_{\substack{\bszero \neq \bsm \leq \bsnu \\ \bsm \neq \bsnu}}
\binom{\bsnu}{\bsm} |\bsm|!^{1 + \upepsilon} \bsbeta^\bsm
\Big[\nrm{\pdy^{\bsnu - \bsm}(u - u_h)}{V}
+ |\pdy^{\bsnu - \bsm}(\lambda - \lambda_h)|\Big]\Bigg),
\end{align}
where $\bsbeta$, $\bsbetabar$ are defined in \eqref{eq:beta}, \eqref{eq:betabar}, respectively,
and $\CIII$ is independent of $\bsy$, $h$ and $\bsnu$.
\end{lemma}

The astute reader may now ask, why do we need both the bounds \eqref{eq:dlam_fe_rec_a} and \eqref{eq:dlam_fe_rec_b} on the derivative of the eigenvalue error? 
The reason is that the upper bound in \eqref{eq:dlam_fe_rec_a} depends only
on derivatives with order strictly less than $\bsnu$, 
whereas the bound in \eqref{eq:dlam_fe_rec_b} depends on $\pdy^\bsnu(u - u_h)$.
Hence the inductive step for the \emph{eigenfunction} (see \eqref{eq:du-fe} below) 
only works with \eqref{eq:dlam_fe_rec_a}.
On the other hand, \eqref{eq:dlam_fe_rec_a} cannot be used for the inductive step for the 
\emph{eigenvalue} result (see \eqref{eq:dlam_fe} below), 
because it will only result in a bound of order $\bigO(h)$.
Hence, the second bound \eqref{eq:dlam_fe_rec_b} is required to maintain
the optimal rate of $\bigO(h^2)$ for the eigenvalue error.

We now have the necessary ingredients to prove the following bounds
on the derivatives of the FE error.

\begin{theorem}
\label{thm:du-fe}
Let $\bsnu \in \indx$ be a multi-index, let $h > 0$ be sufficiently small and suppose that Assumption~A\ref{asm:coeff} holds.
Define the sequence $\bsbetahat = (\betahat_j)_{j \in \N}$ by
\begin{equation}
\label{eq:betahat}
\betahat_j \,\coloneqq\, \widehat{C}_\bsbeta 
\max\big(\nrm{a_j}{L^\infty}, \nrm{b_j}{L^\infty}, \nrm{\nabla a_j}{L^\infty}\big),
\end{equation}
where $\Chat_\bsbeta$, given explicitly below in \eqref{eq:C_hat}, is
independent of $\bsy$, $h$ and $j$.
Then 
\begin{align}
\label{eq:dlam_fe}
\big|\pdy^\bsnu\big[\lambda(\bsy)  - \lambda_h(\bsy)\big]\big| 
\,&\leq\,
C_1\,  |\bsnu|!^{1 + \upepsilon}\bsbetahat^\bsnu h^2,\\
\label{eq:du-fe}
\nrm{\pdy^\bsnu\big[u(\bsy) - u_h(\bsy)\big]}{V} 
\,&\leq\, 
C_2\, |\bsnu|!^{1 + \upepsilon} \bsbetahat^\bsnu h,
\end{align}
and for $\calG \in H^{-1 + t}(D)$
\begin{equation}
\label{eq:dGu-fe}
\big|\pdy^\bsnu\calG (u(\bsy) - u_h(\bsy))\big| \,\leq\, 
C_3 \,
|\bsnu|!^{1 + \upepsilon} \bsbetahat^\bsnu h^{1 + t},
\end{equation}
with $C_1, C_2, C_3$ all independent of $\bsy$, $h$ and $\bsnu$.
\end{theorem}

\begin{proof}
Throughout we use the convention that $0! = 1$. 
Then, due to the error bound \eqref{eq:fe_lam} for the FE eigenvalue error, 
the base case of the induction 
($\bsnu = \bs0$) for the eigenvalue result \eqref{eq:dlam_fe} holds provided 
$C_1 \geq C_\lambda$.
Thus, let
\begin{equation}
\label{eq:C_dlam}
C_1 \,\coloneqq\, \max\big\{C_{\lambda}, 
\; \CII C_{u}^2 (2 + C_\upepsilon)(4 + C_\upepsilon) \big\}.
\end{equation}
Similarly, defining $C_2 = C_{u}$ the base case of the induction for \eqref{eq:du-fe} 
also clearly holds due to \eqref{eq:fe_u}.
 
For the inductive step, let $\bsnu$ be such that $|\bsnu| \geq 1$ and
assume that \eqref{eq:dlam_fe} and \eqref{eq:du-fe} hold for all $\bsm$ with $|\bsm| < |\bsnu|$.
Now, since the recursive bound for the eigenvalue \eqref{eq:dlam_fe_rec_b}
still depends on a term of order $\bsnu$, whereas the recursive bound for the eigenfunction
\eqref{eq:du_fe_rec} only depends on strictly lower order terms,
for our inductive step to work we first prove the result \eqref{eq:du-fe} 
for the eigenfunction, before proving the result \eqref{eq:dlam_fe} for the eigenvalue.

Substituting the induction assumptions \eqref{eq:dlam_fe} and \eqref{eq:du-fe} 
for $|\bsm| < |\bsnu|$ into \eqref{eq:du_fe_rec} gives
\begin{align*}
\nrm{\pdy^\bsnu (u - u_h)}{V} 
\,\leq\, &
\CIII \Bigg( |\bsnu|!^{1 + \upepsilon} \bsbetabar^\bsnu h
+
\sum_{j = 1}^\infty \nu_j \beta_j
C_2 (|\bsnu| - 1)!^{1 + \upepsilon} \bsbetahat^{\bsnu - \bse_j} h
\\
&+\,
\sum_{\substack{\bszero \neq \bsm \leq \bsnu \\ \bsm \neq \bsnu}}
\binom{\bsnu}{\bsm} 
|\bsm|!^{1 + \upepsilon} \bsbeta^{\bsm}
(C_2 + C_1 \hsuff)\ |\bsnu - \bsm|!^{1 + \upepsilon}\ 
\bsbetahat^{\bsnu - \bsm} h \Bigg)
\\
\leq\, &
 \bsbetahat^\bsnu h \,\CIII \Bigg(
\bigg[ \bigg(\frac{C_{\bsbeta}}{\Chat_\bsbeta}\bigg)^{|\bsnu|}
+ \frac{C_2 C_\bsbeta}{\Chat_\bsbeta}\bigg] 
|\bsnu|!^{1 + \upepsilon}
\\
&
+\,
\sum_{\substack{\bszero \neq \bsm \leq \bsnu \\ \bsm \neq \bsnu}}
\binom{\bsnu}{\bsm} |\bsnu - \bsm|!^{1 + \upepsilon} |\bsm|!^{1 + \upepsilon} 
(C_1\hsuff + C_2)\bigg(\frac{C_{\bsbeta}}{\Chat_\bsbeta}\bigg)^{|\bsm|}
\Bigg)
\end{align*}
where we have rescaled each product by using the definitions of $\bsbeta$, 
$\bsbetabar$ and $\bsbetahat$
in \eqref{eq:beta}, \eqref{eq:betabar} and \eqref{eq:betahat}.
The constants can be simplified by defining
\begin{equation}
\label{eq:C_hat}
\Chat_\bsbeta \,\coloneqq\, 
\frac{1}{C_2} C_\bsbeta \max(1, \CIII) \big[1 + C_2 + C_\upepsilon (C_1\hsuff + C_2)\big],
\end{equation}
which is independent of $\bsy$, $h$ and $\bsnu$.
This guarantees that $C_{\bsbeta}/ \Chat_\bsbeta \leq 1$, 
and thus since $|\bsm|, |\bsnu - \bsm| \geq 1$, we have the bound
\begin{align*}
\big\|\pdy^\bsnu &(u - u_h)\big\|_V
\\
&\leq\, 
 \bsbetahat^\bsnu h \,\CIII \frac{C_{\bsbeta}}{\Chat_\bsbeta} \Bigg(
(1 + C_2) |\bsnu|!^{1 + \upepsilon}
+\,
  (C_1\hsuff + C_2)
\sum_{\substack{\bszero \neq \bsm \leq \bsnu \\ \bsm \neq \bsnu}}
\binom{\bsnu}{\bsm} |\bsnu - \bsm|!^{1 + \upepsilon} |\bsm|!^{1 + \upepsilon} 
\Bigg)
\\
&\leq\,
\CIII \frac{C_{\bsbeta}}{\Chat_\bsbeta} \big[
1 + C_2 +  C_\upepsilon (C_1\hsuff + C_2) \big]\,
 |\bsnu|!^{1 + \upepsilon} \bsbetahat^\bsnu h
 \,\leq\, C_2 |\bsnu|!^{1 + \upepsilon} \bsbetahat^\bsnu h,
\end{align*}
where we have used \cite[Lemma~3.3]{GGKSS19} to bound the sum from above by
$C_\upepsilon |\bsnu|!^{1 + \upepsilon}$ (see also \eqref{eq:sum_bnd}), as well as
\eqref{eq:C_hat} to give the final result.

For the inductive step for the eigenvalue, we substitute the result \eqref{eq:du-fe},
which has just been shown to hold for all multi-indices of order up to and including $|\bsnu|$, 
into \eqref{eq:dlam_fe_rec_b} and then simplify, to give
\begin{align*}
|\pdy^\bsnu(\lambda - \lambda_h)| \,\leq\,
&\CII (C_2)^2 \bsbetahat^\bsnu h^2
\Bigg( \sum_{\bsm \leq\bsnu} \binom{\bsnu}{\bsm}
 |\bsnu - \bsm|!^{1 + \upepsilon} |\bsm|!^{1 + \upepsilon}
\\
&+ 
\sum_{j = 1}^\infty \nu_j \sum_{\bsm \leq \bsnu - \bse_j} 
\binom{\bsnu - \bse_j}{\bsm} |\bsnu - \bse_j - \bsm|!^{1 + \upepsilon} |\bsm|!^{1 + \upepsilon}
\\
&+ 
\sum_{\bsm \leq \bsnu} \binom{\bsnu}{\bsm} |\bsnu - \bsm|!^{1 + \upepsilon}
\sum_{\bsk \leq \bsm}
\binom{\bsm}{\bsk}  |\bsm - \bsk|!^{1 + \upepsilon} |\bsk|!^{1 + \upepsilon}\Bigg),
\end{align*}
where we have used the fact that $\beta_j \leq \betahat_j$. The sums can again 
be bounded using \eqref{eq:sum_bnd} (using it twice for the double sum on the last line),
 to give
\begin{align*}
|\pdy^\bsnu(\lambda - \lambda_h)| \,\leq\,
\CII (C_2)^2 \,\bsbetahat^\bsnu h^2\,(2 + C_\upepsilon)(4 + C_\upepsilon)\,
|\bsnu|!^{1 + \upepsilon},
\end{align*}
which, with $C_1$ as defined in \eqref{eq:C_dlam} and $C_2 = C_{u}$,
gives our desired result \eqref{eq:dlam_fe}.

The final result for the derivative of the error of the linear functional $\calG(u)$ \eqref{eq:dGu-fe}
follows by considering the same dual problem (A.16) as in \cite{GGKSS19}.
But instead, here we let $w = \pdy^\bsnu(u - u_h)$ and then use the upper bound
\eqref{eq:du-fe} in the last step.
\end{proof}

We can now simply substitute these bounds on the derivatives of the FE error into
\eqref{eq:err_ell_decomp}, in order to bound the FE component of the error.
For the second and third term in \eqref{eq:err_ell_decomp}, 
in the case of the eigenvalue, this gives
\begin{equation}
\label{eq:fe_W_lam}
\nrm{\lambda_s(\bsy) - \lambda_{h, s}(\bsy)}{\Ws}
\,\leq\, h^2 \Bigg(C_1^2 \sum_{\setu \subseteq \{1:s\}} 
\frac{|\setu|!^{2(1 + \upepsilon)}}{\gamma_\setu} \prod_{j \in \setu} \betahat_j^2\Bigg)^{1/2}.
\end{equation}
Similar results hold for $u(\bsy)$ and $\calG(u(\bsy))$.

To ensure that the constant on the RHS of \eqref{eq:fe_W_lam}, 
and the constants in the bounds that follow, are
independent of the dimension, for $\xi \in (\frac{1}{2}, 1]$ to be specified later,
we will choose the weights $\bsgamma$ by
\begin{equation}
\label{eq:gamma}
\gamma_j \,=\, \max\big( \betahat_j, \beta_j^{p/q}\big),
\qquad
\gamma_\setu \,=\, \Bigg( (|\setu| + 3)!^{2(1 + \upepsilon)} 
\prod_{j \in \setu} \frac{(2\pi^2)^\xi}{2\zeta(2\xi)}\gamma_j^2\Bigg)^{1/(1 + \xi)},
\end{equation}
where $p$, $q$ are the summability parameters from 
Assumption~A\ref{asm:coeff}.\ref{itm:summable}, so that 
$(\gamma_j)_{j \in \N} \in \ell^q(\R)$.

\subsection{Estimating the truncation error}
\label{sec:truncerr}
It remains to estimate the first term in \eqref{eq:err_ell_decomp} --- the truncation error.
\begin{theorem}
\label{thm:trunc-err}
Suppose that Assumption~A\ref{asm:coeff} holds and let $s, \widetilde{s} \in \N$ with $s > \widetilde{s}$.
Additionally, suppose that 
the weights $\bsgamma$ are given by \eqref{eq:gamma},
then
\begin{equation}
\label{eq:trunc_W_lam}
\nrm{\lambda_s - \lambda_{\widetilde{s}}}{\Ws} \,\lesssim\,
\widetilde{s}^{\,-1/p + 1/q} 
\Bigg( \sum_{\setu \subseteq \{1:\widetilde{s}\}}
\frac{(|\setu| + 3)!^{2(1 + \upepsilon)}}{\gamma_\setu} \prod_{j \in \setu} \beta_j^2
\Bigg)^{1/2},
\end{equation}
with the constant independent of $\widetilde{s}$ and $s$.
\end{theorem}

\begin{proof}
Since $\lambda_s$ is analytic we can expand it as a Taylor series
about $\bszero$ in the variables $\{y_{\widetilde{s} + 1}, \ldots, y_s\}$:
\[
\lambda_s(\bsy_s) \,=\,
\lambda_s(\bsy_{\widetilde{s}}; \bszero) + \sum_{i = \widetilde{s} + 1}^s y_i
\int_0^1 \pd{}{}{y_i} \lambda_s(\bsy_{\widetilde{s}}; t\bsy_{\{\widetilde{s} + 1:s\}}) \, \rd t,
\]
where we use the notation $\bsy_s = (y_1, y_2, \ldots, y_s)$, 
$(\bsy_{\widetilde{s}}; \bszero) = (y_1, y_2, \ldots, y_{\widetilde{s}}, 0, \ldots 0)$
and $(\bsy_{\widetilde{s}}; t\bsy_{\{\widetilde{s} + 1:s\}}) = (y_1, y_2, \ldots, y_{\widetilde{s}}, 
t y_{\widetilde{s} + 1}, t y_{\widetilde{s} + 2}, \ldots, t y_s)$.

Since $\lambda_{\widetilde{s}}(\bsy_{\widetilde{s}}) = \lambda_s(\bsy_{\widetilde{s}}; \bszero)$
(this is simply different notation for the same object),
 this can be rearranged to give
\begin{equation}
\label{eq:truncerr_exact}
\lambda_s(\bsy_s) - \lambda_{\widetilde{s}}(\bsy_{\widetilde{s}}) \,=\, 
\sum_{i = \widetilde{s} + 1}^s y_i
\int_0^1 \pd{}{}{y_i}(\bsy_{\widetilde{s}}; t\bsy_{\{\widetilde{s} + 1:s\}}) \, \rd t.
\end{equation}

Let $\setu \subseteq \{1, 2\ldots, \widetilde{s}\}$, then differentiating
\eqref{eq:truncerr_exact} with respect to $\bsy_\setu$ gives
\[
\pd{|\setu|}{}{\bsy_\setu}\big(\lambda_s(\bsy_s) - \lambda_{\widetilde{s}}(\bsy_{\widetilde{s}})\big) 
\,=\, 
\sum_{i = \widetilde{s} + 1}^s y_i
\int_0^1 \pd{|\setu| + 1}{}{\bsy_{\setu \cup \{i\}}} \lambda_s(\bsy_{\widetilde{s}}; t\bsy_{\{\widetilde{s} + 1:s\}}) \, \rd t.
\]
Taking the absolute value, using the triangle inequality and the fact that $|y_j| \leq 1/2$,
we have the upper bound
\[
\bigg|\pd{|\setu|}{}{\bsy_\setu}\big(\lambda_s(\bsy_s) - \lambda_{\widetilde{s}}(\bsy_{\widetilde{s}})\big)\bigg|
\,\leq\, 
\frac{1}{2}\sum_{i = \widetilde{s} + 1}^s
\int_0^1 \bigg|\pd{|\setu| + 1}{}{\bsy_{\setu \cup \{i\}}} \lambda_s(\bsy_{\widetilde{s}}; t\bsy_{\{\widetilde{s} + 1:s\}})\bigg| 
\, \rd t.
\]
Now, substituting in the upper bound on the derivative of $\lambda_s$
from \cite[Lemma~3.4, equation (3.6)]{GGKSS19} gives
\begin{align}
\label{eq:truncbnd-1}
\nonumber
\bigg|\pd{|\setu|}{}{\bsy_\setu}\big(\lambda_s(\bsy_s) - \lambda_{\widetilde{s}}(\bsy_{\widetilde{s}})\big)\bigg|
\,&\leq\, 
\frac{\lambdabar}{2} \sum_{i = \widetilde{s} + 1}^s (|\setu| + 1)!^{1 + \upepsilon} \beta_i
\prod_{j \in \setu} \beta_j\\
\,&=\, \frac{\lambdabar}{2} \Bigg(\sum_{i = \widetilde{s} + 1}^s  \beta_i \Bigg)
(|\setu| + 1)!^{1 + \upepsilon} \prod_{j \in \setu} \beta_j,
\end{align}
with $\beta_j$ as in \eqref{eq:beta}.

Letting $\setu \subseteq \{1, 2, \ldots, s\} $ with $\setu \cap \{\widetilde{s} + 1, \widetilde{s} + 2, \ldots, s\} \neq \emptyset$,
the derivative $\lambda_s - \lambda_{\widetilde{s}}$ is simply
\begin{equation}
\label{eq:truncbnd-2}
\bigg|\pd{|\setu|}{}{\bsy_\setu}\big(\lambda_s(\bsy_s) - \lambda_{\widetilde{s}}(\bsy_{\widetilde{s}})\big)\bigg|
\,=\, \bigg|\pd{|\setu|}{}{\bsy_\setu} \lambda_s(\bsy_s) \bigg|
\,\leq\, \frac{\lambdabar}{2} |\setu|!^{1 + \upepsilon} \prod_{j \in \setu} \beta_j,
\end{equation}
where we have again used the upper bound \cite[equation (3.6)]{GGKSS19}.

We now bound the norm \eqref{eq:W-norm} of $\lambda_s - \lambda_{\widetilde{s}}$ in $\Ws$. 
Splitting the sum over $\setu \subseteq \{1, 2, \ldots\}$ by whether $\setu$
contains any of $\{\widetilde{s} + 1, \widetilde{s} + 2, \ldots, s\}$, we can write 
\begin{align*}
\|\lambda_s - &\lambda_{\widetilde{s}}\|_{\Ws}^2
= \sum_{\setu \subseteq \{1:\widetilde{s}\}} \frac{1}{\gamma_\setu} \int_{[-\frac{1}{2}, \frac{1}{2}]^{|\setu|}} 
\bigg(\int_{[-\frac{1}{2}, \frac{1}{2}]^{s - |\setu|}} \pd{|\setu|}{}{\bsy_\setu} 
\big[\lambda_s(\bsy_s) - \lambda_{\widetilde{s}}(\bsy_{\widetilde{s}})\big] \, \rd \bsy_{-\setu}\bigg)^2
\rd \bsy_\setu\\
&+ \sum_{\substack{\setu \subseteq \{1:s\}\\ \setu \cap \{\widetilde{s} + 1:s\} \neq \emptyset}}
\frac{1}{\gamma_\setu} \int_{[-\frac{1}{2}, \frac{1}{2}]^{|\setu|}} 
\bigg(\int_{[-\frac{1}{2}, \frac{1}{2}]^{s - |\setu|}} \pd{|\setu|}{}{\bsy_\setu} 
\big[\lambda_s(\bsy_s) - \lambda_{\widetilde{s}}(\bsy_{\widetilde{s}})\big] \, \rd \bsy_{-\setu}\bigg)^2
\, \rd \bsy_\setu.
\end{align*}
Substituting in the bounds \eqref{eq:truncbnd-1} and \eqref{eq:truncbnd-2} then yields
\begin{align}
\label{eq:truncerr_sum}
\nonumber
\nrm{\lambda_s - \lambda_{\widetilde{s}}}{\Ws} 
\,\leq\,&
\Bigg[\Bigg(\sum_{i = \widetilde{s} + 1}^s \beta_i\Bigg)^2
\sum_{\setu \subseteq \{1:\widetilde{s}\}} \frac{(|\setu| + 1)!^{2(1 + \upepsilon)}}{\gamma_\setu}
\prod_{j \in \setu} \beta_j^2
\\
&\quad+ \sum_{\substack{\setu \subseteq \{1:s\}\\\setu \cap \{\widetilde{s} + 1:s\} \neq \emptyset}}
\frac{|\setu|!^{2(1 + \upepsilon)}}{\gamma_\setu} \prod_{j \in \setu} \beta_j^2\Bigg]^{1/2}.
\end{align}

Next, we can bound the sum over $i$ in
\eqref{eq:truncerr_sum} by the whole tail of the sum, which can then be bounded using 
\cite[eq. (4.7)]{GGKSS19}, to give
\begin{equation}
\label{eq:tail_sum}
\sum_{i = \widetilde{s} + 1}^s \beta_i 
\,\leq\,
\sum_{i = \widetilde{s} + 1}^\infty \beta_i
\,\leq\, \min \bigg( \frac{p}{1 - p}, 1\bigg) \nrm{\bsbeta}{\ell^p} \widetilde{s}^{\,- (1/p - 1)}.
\end{equation}
Then for weights given by \eqref{eq:gamma}, following the proof of \cite[Theorem 11]{KSS15}
we can bound 
\begin{equation}
\label{eq:trunc_weights}
\sum_{\substack{\setu \subseteq \{1:s\}\\\setu \cap \{\widetilde{s} + 1:s\} \neq \emptyset}}
\frac{|\setu|!^{2(1 + \upepsilon)}}{\gamma_\setu} \prod_{j \in \setu} \beta_j^2
\,\lesssim\,
\widetilde{s}^{-2(1/p - 1/q)} \sum_{\setu \subseteq \{1:\widetilde{s}\}}
\frac{(|\setu| + 3)!^{2(1 + \upepsilon)}}{\gamma_\setu} \prod_{j \in \setu} \beta_j^2,
\end{equation}
with a constant that is independent of $\widetilde{s}$ and $s$.

Since $1/q > 1$, after substituting the bounds \eqref{eq:tail_sum} and
\eqref{eq:trunc_weights} into \eqref{eq:truncerr_sum} we obtain the final result
 \eqref{eq:trunc_W_lam}.
\end{proof}

\subsection{Final error bound}
In the previous two sections we have successfully bounded the FE and truncation
error in the $\Ws$ norm, now these bounds can simply be substituted into 
\eqref{eq:err_ell_decomp} to bound the variance on each level.

\begin{theorem}
\label{thm:Var_ell}
Let $L \in \N$, let $1 = h_{-1} > h_0 >  h_1 > \cdots > h_L > 0$ with $h_0$ sufficiently small,
let $1 = s_{-1} = s_0 \leq s_1 \cdots \leq \cdots s_L$,
and suppose that Assumption~A\ref{asm:coeff} holds with $p < q$.
Also, let each $Q_\ell$ be a lattice rule using $N_\ell = 2^{n_\ell}$, $n_\ell \in \N$, points
corresponding to a CBC-constructed generating vector
with weights $\bsgamma$ given by \eqref{eq:gamma}.
Then, for all $\ell = 0, 1, 2, \ldots L$,
\begin{equation}
\label{eq:Var_ell_lam}
\bbV_\bsDelta[Q_\ell(\lambda_\ell - \lambda_{\ell - 1})] \,\leq\, C_1 N_\ell^{-\eta} 
\big(h_{\ell - 1}^4  +  s_{\ell - 1}^{-2(1/p - 1/q)}\big),
\end{equation}
and, for $\calG \in H^{-1 + t}(D)$ with $0 \leq t \leq 1$,
\begin{equation}
\label{eq:Var_ell_G}
\bbV_\bsDelta[Q_\ell(\calG(u_\ell) - \calG(u_{\ell - 1}))] \,\leq\, C_2 N_\ell^{-\eta} 
\big(h_{\ell - 1}^{2(1 + t)}  + s_{\ell - 1}^{-2(1/p - 1/q)}\big),
\end{equation}
where, for $0 < \delta < 1$,
\begin{equation*}
\eta \,=\, \begin{cases}
2 - \delta & \text{if } q \in (0, \tfrac{2}{3}],\\[2mm]
\displaystyle{\frac{2}{q} - 1} & \text{if } q \in (\tfrac{2}{3}, 1).
\end{cases}
\end{equation*}
The second term in \eqref{eq:Var_ell_lam} and \eqref{eq:Var_ell_G} can be dropped
if $s_\ell = s_L$, for $\ell = 1, 2, \ldots L$.
\end{theorem}

\begin{proof}
We prove the result for the eigenvalue, since the eigenfunction result
follows analogously.
For $\ell \geq 1$, substituting the bounds \eqref{eq:fe_W_lam} and \eqref{eq:trunc_W_lam}
into \eqref{eq:err_ell_decomp} gives
\[
\nrm{\lambda_\ell - \lambda_{\ell - 1}}{\Ws}
\,\leq\, 
 C \Bigg( \sum_{\setu \subseteq \{1:s_{\ell}\}}
\frac{(|\setu| + 3)!^{2(1 + \upepsilon)}}{\gamma_\setu} 
\prod_{j \in \setu} \betahat_j^2\Bigg)^{1/2}
\big( h_{\ell - 1}^2 + s_{\ell - 1}^{-(1/p - 1/q)}\big) ,
\]
where we have simplified by using that $h_{\ell} < h_{\ell - 1}$, $s_{\ell - 1} < s_\ell$, 
$\beta_j \leq \betahat_j$, and also merged all constants into a generic constant $C$,
which may depend on $\upepsilon$.

Substituting the bound above into \eqref{eq:Var_ell},
then using that $N_\ell = 2^{n_\ell}$ and thus $\varphi(N_\ell) = N_\ell/2$,
the variance on level $\ell$ can be bounded by
\begin{equation}
\label{eq:Var_ell_a}
\bbV_\bsDelta[Q_\ell(\lambda_\ell - \lambda_{\ell - 1})] 
\,\leq\, 
C_{\ell, \bsgamma, \xi}\,
N_\ell^{-1/\xi} \big(s_{\ell - 1}^{-2(1/p - 1/q)} + h_{\ell - 1}^4\big).
\end{equation}
The constant is given by 
\[
C_{\ell, \bsgamma, \xi} \,\coloneqq\,
C^2 2^{1/\xi}\Bigg( \sum_{\setu \subseteq \{1:s_\ell\}}
\frac{(|\setu| + 3)!^{2(1 + \upepsilon)}}{\gamma_\setu} 
\prod_{j \in \setu} \betahat_j^2\Bigg)
 \Bigg(\sum_{\emptyset \neq \setu \subseteq \{1:s_\ell\}}
\gamma_\setu^\xi \left( \frac{2\zeta(2\xi)}{(2\pi^2)^\xi}\right)^{|\setu|}
\Bigg)^{1/\xi}.
\]

For $\ell = 0$ we can similarly substitute \eqref{eq:dlambda} into the CBC bound \eqref{eq:cbc_err},
and then since $h_{-1} = s_{-1} = 1$ and $\beta_j \leq \betahat_j$ 
it follows that \eqref{eq:Var_ell_a} also holds for $\ell = 0$ with the 
constant $C_{0, \bsgamma, \xi}$ as above.

All that remains to be shown is that this constant can be bounded independently of 
$s_{\ell - 1}$ and $s_\ell$.
To this end, substituting the formula \eqref{eq:gamma} for $\gamma_\setu$ and using the fact that $\betahat_j \leq \gamma_j$ then simplifying, we can bound 
$C_{\ell, \bsgamma, \xi}$ above by
\[
C_{\ell, \bsgamma, \xi} 
\,\leq\, C \Bigg(
\underbrace{\sum_{|\setu| < \infty} (|\setu| + 3)!^\frac{2\xi(1 + \upepsilon)}{1 + \xi}
\prod_{j \in \setu} \gamma_j^\frac{2\xi}{1 + \xi} 
\bigg(\frac{2\zeta(2\xi)}{(2\pi^2)^\xi} \bigg)^\frac{1}{1 + \xi}
}_{S_{\xi}}\Bigg)^\frac{1 + \xi}{\xi},
\]
where again $C$ is a generic constant, which may depend on $\upepsilon$.

We now choose the exponents $\xi$ and $\upepsilon$ so that the sum $S_\xi$ is finite. 
For $0 < \delta < 1$, let
\begin{equation}
\label{eq:xi}
\xi \,=\,
\begin{cases}
{\displaystyle \frac{1}{2 - \delta}} & \text{if } q \in (0, \frac{2}{3}],\\[2mm]
{\displaystyle \frac{q}{2 - q} } & \text{if } q \in (\frac{2}{3}, 1),
\end{cases}
\quad \text{and} \quad
\upepsilon \,=\, \frac{1 - \xi}{4\xi} \,>\, 0.
\end{equation}
With this choice of $\xi$ we have
$(4\xi - q - 3q\xi)/(1 - \xi) \geq q$ for any $q \in (0, 1)$, and so
\begin{equation}
\label{eq:sum_gamma}
\sum_{j = 1}^\infty \gamma_j^\frac{4\xi - q - 3q\xi}{1 - \xi} \,<\, \infty.
\end{equation}
Then define the sequence
\begin{equation}
\label{eq:alpha}
\alpha_j \,=\, \Bigg( 1 + \sum_{i = 1}^\infty \gamma_i^q\Bigg)^{-1} \gamma_i^q,
\quad \text{so that} \quad
\sum_{j = 1}^\infty \alpha_j \,<\, 1.
\end{equation}

Substituting in our choice \eqref{eq:xi} for $\upepsilon$, and multiplying and dividing
each term by the product of $\alpha_j^{(1 + 3\xi)/(2(1 + \xi))}$, we can write 
\[
S_{\xi} \,=\, \sum_{|\setu| < \infty} (|\setu| + 3)!^\frac{1 + 3\xi}{2(1 + \xi)}
\Bigg(\prod_{j \in \setu} \alpha_j^\frac{1 + 3\xi}{2(1 + \xi)}\Bigg)
\Bigg(\prod_{j \in \setu} 
\gamma_j^\frac{2\xi}{1 + \xi}\alpha_j^{-\frac{1 + 3\xi}{2(1 + \xi)}}
\bigg(\frac{2\zeta(2\xi)}{(2\pi^2)^\xi} \bigg)^\frac{1}{1 + \xi}
\Bigg).
\]
Applying H\"older's inequality with exponents 
$2(1 + \xi)/(1 + 3\xi) > 1$ and $2(1 + \xi)/(1 - \xi) > 1$ gives
\begin{align*}
S_\xi \,\leq\, &\Bigg( \sum_{|\setu| < \infty}
(|\setu| + 3)! \prod_{j \in \setu} \alpha_j\Bigg)^\frac{1 + 3\xi}{2(1 + \xi)}
\Bigg( \sum_{|\setu| < \infty} \prod_{j \in \setu} \gamma_j^\frac{4\xi}{1 - \xi}
\alpha_j^{-\frac{1 + 3\xi}{1 - \xi}}
\bigg(\frac{2\zeta(2\xi)}{(2\pi^2)^\xi} \bigg)^\frac{2}{1 - \xi}
\Bigg)^\frac{2 - \xi}{2(1 + \xi)}
\\
\leq\,&
\Bigg[ 6 \bigg(1 - \sum_{j = 1}^\infty \alpha _j\bigg)^{-4}\Bigg]^\frac{1 + 3\xi}{2(1 + \xi)}
\\
&\cdot\exp \Bigg[ \frac{2 - \xi}{2(1 + \xi)} 
\bigg(\frac{2\zeta(2\xi)}{(2\pi^2)^\xi} \bigg)^\frac{2}{1 - \xi}
\bigg(1 + \sum_{j = 1}^\infty \gamma_j^q\bigg)^\frac{1 + 3\xi}{1 - \xi}
\sum_{j = 1}^\infty \gamma_j^\frac{4\xi - q - 3q\xi}{1 - \xi} \Bigg],
\end{align*}
where we have used \cite[Lemma~6.3]{KSS12}. From \eqref{eq:sum_gamma}
and \eqref{eq:alpha} it follows that $S_\xi < \infty$, and so
$C_{\ell, \bsgamma, \xi}$ can be bounded independently of $s_\ell$.
Finally, letting $\eta = 1/\xi$ for $\xi$ as in \eqref{eq:xi} gives the desired result
with a constant independent of $s_\ell$.
\end{proof}

\begin{remark}\label{rem:err-params}
Hence, we have verified that Assumptions M2 from 
Theorems~\ref{thm:ML_abstract_lam} and \ref{thm:ML_abstract_G} hold with
$\beta_\lambda = 2\alpha_\lambda = 4$, $\beta_\calG = 2\alpha_\calG = 2(1 + t)$, 
$\beta' = 1/p + 1/q$, and $\eta$ as given above.
\end{remark}

The upper bounds in Theorem~\ref{thm:du-fe}, \eqref{eq:fe_W_lam} and 
Theorem~\ref{thm:trunc-err} are the same as the corresponding bounds from
the MLQMC analysis for the source problem (see \cite[Theorems 7, 8, 11]{KSS15}),
the only differences are in the values of the constants and in the extra $1 + \upepsilon$
factor in the exponent of $|\setu|!$. As such the final variance bounds in 
Theorem~\ref{thm:Var_ell} also coincide with the bounds for the source problem from
\cite{KSS15} for all $q < 1$. The only difference is that our result does not hold for $q = 1$,
whereas the results for the source problem do.

\subsection{Extension to higher-order QMC}
\label{sec:hoqmc}
As mentioned earlier, the bounds on the higher-order derivatives
that we proved in Section~\ref{sec:reg} imply
higher order methods can also be used for the quadrature component of
our ML algorithm, which will provide a faster convergence rate in $N_\ell$.
We now provide a brief discussion
of how to extend our ML algorithm, and the error analysis, to 
\emph{higher-order QMC (HOQMC)} rules.
From an algorithm point of view, one can simply use HOQMC points instead of lattice rules
for the quadrature rules $Q_\ell$ in \eqref{eq:mlqmc0}.
We denote this ML-HOQMC approximation by $Q^\mathrm{MLHO}_L$.
To extend the error analysis to HOQMC we can again use a general framework
as in Theorems~\ref{thm:ML_abstract_lam} and \ref{thm:ML_abstract_G}.
We stress that the difficult part is to  verify the assumptions, 
and in particular to show the required mixed higher-order derivative bounds that we 
have already proved in Theorem~\ref{thm:regularity_y}.
The remainder of the analysis then follows the same steps as in the previous sections 
with only slight modifications to handle the higher-order norm as in \cite{DKLeGS16},
where ML-HOQMC methods were applied to PDE source problems.
As such, we don't present the full details here but only an outline.

A HOQMC rule is an equal-weight quadrature rule of the form \eqref{eq:rqmc}
that can achieve faster than $1/N$ convergence for sufficiently smooth integrands.
A popular class of deterministic HOQMC rules are \emph{interlaced polynomial lattice rules},
see \cite{Dick08,GodaDick15} and \cite{DKLeGNS14,DKLeGS16} for their application
to PDE source problems. 
Loosely speaking, a polynomial lattice rule is a QMC rule similar to a lattice rule,
except the points are generated by a vector of polynomials instead
of integers, the number of points $N$ is a prime power and the points
are not randomly shifted. Higher order convergence in $s$ dimensions
is then achieved by taking a polynomial lattice rule in a higher dimension, 
$\nu \cdot s$ for $\nu \in \N$,
and cleverly interlacing the digits across the dimensions of each $(\nu s)$-dimensional point to
produce an $s$-dimensional point. The factor $\nu \in \N$ is
called the \emph{interlacing order} and it determines the convergence rate.
Good interlaced polynomial lattice rules can also be constructed by a CBC algorithm.
See \cite{GodaDick15} for the full details.

Following \cite{DKLeGS16}, for $\nu \in \N$ and $1 \leq r \leq \infty$ 
we introduce the Banach space $\calW^{\nu, r}_{s, \bsgamma}$, 
which is a higher-order analogue of the first-order space $\calW_{s, \bsgamma}$, with the norm
\begin{align}
\label{eq:W^nu}
\|f\|_{\calW^{\nu, r}_{s, \bsgamma}} 
\,=\, &\max_{\setu \subseteq \{1:s\}} \frac{1}{\gamma_\setu}
\Bigg(\sum_{\setv \subseteq \setu} 
\sum_{\bstau_{\setu \setminus \setv} \in \{1:\nu\}^{|\setu \setminus \setv|}}
\nonumber\\
&\int_{[-\frac{1}{2}, \frac{1}{2}]^{|\setv|}}
\bigg| \int_{[-\frac{1}{2}, \frac{1}{2}]^{s - |\setv|}}
\partial^{(\bsnu_\setv, \bstau_{\setu \setminus \setv}, \bszero)}_\bsy f(\bsy)
\, \rd \bsy_{-\setv} \bigg|^r\rd \bsy_\setv\Bigg)^{1/r}.
\end{align}
Here $(\bsnu_\setv, \bstau_{\setu \setminus \setv}, \bszero) \in \calF$ is the multi-index
with $j$th entry given by $\nu$ if $j \in \setv$, $\tau_j$ if $j \in \setu \setminus \setv$
and 0 otherwise. For $f \in \calW_{s, \bsgamma}^{\nu, r}$, an order $\nu$
interlaced polynomial lattice rule using $N$ points in $s$ dimensions
can be constructed using a CBC algorithm
such that the (deterministic) error converges at a rate $N^{-\eta}$ for $1\leq \eta < \nu$
(see \cite[Theorem~3.10]{DKLeGNS14}).

Let $\nu_p = \lfloor 1/p\rfloor + 1$ for $p< 1$ as in Assumption~A\ref{asm:coeff} and 
$1 \leq r \leq \infty$, 
then it follows from \eqref{eq:dlambda} that $\lambda_s \in \calW_{s, \bsgamma}^{\nu_p, r}$
for all $s$. Hence, the error of a single level QMC approximation of $\bbE_\bsy[\lambda_s]$ using
an order $\nu_p$ interlaced polynomial lattice rule will converge as $N^{-1/p}$.
Similarly, the ML analysis can be extended to show that a ML-HOQMC method
achieves higher order convergence in $N_\ell$,
where in this case we choose the interlacing factor to be $\nu_q = \lfloor 1/q \rfloor + 1$
for $q < 1$ as in Assumption~A\ref{asm:coeff}.
Indeed, \eqref{eq:dlam_fe} implies that the bound \eqref{eq:fe_W_lam} 
can easily be extended to $\calW_{s, \bsgamma}^{\nu_q, r}$
and \eqref{eq:dlambda} implies that \eqref{eq:trunc_W_lam} can also be extended
to $\calW_{s, \bsgamma}^{\nu_q, r}$ for all $s$. In both cases, the sums over $\setu$ on the
right hand sides need to be updated to account for the form of \eqref{eq:W^nu}, 
but the exponents of $h$ and $s$ remain the same.
Hence, by following the proof of Theorem~\ref{thm:Var_ell} it can be shown that
the following deterministic analogue of the variance bound \eqref{eq:Var_ell} holds\
for interlaced polynomial lattice rules.

\begin{theorem}
\label{thm:ho-diff}
Suppose that Assumption~A\ref{asm:coeff} holds with $p < q < 1$. For $\ell \in \N$,
let $Q_\ell^\mathrm{HO}$ be an interlaced polynomial lattice rule,
constructed using a CBC algorithm with 
$N_\ell$ a prime power number of points and interlacing factor $\nu_q = \lfloor 1/q \rfloor + 1$.
Then $Q^\mathrm{HO}_\ell$ satisfies
\begin{equation}
\label{eq:ho-diff}
\big|Q_\ell^\mathrm{HO}(\lambda_\ell - \lambda_{\ell - 1})\big|
\,\lesssim\, N_\ell^{-1/q}\big(h_\ell^2 + s_\ell^{-1/p + 1/q}\big),
\end{equation}
where the implied constant is independent of $h_\ell$, $s_\ell$ and $N_\ell$.
\end{theorem}

The fact that the implied constant in \eqref{eq:ho-diff} is independent of $s_\ell$
can be shown by following similar arguments as in \cite{DKLeGS16} 
using a special form of $\gamma_\setu$ called 
\emph{smoothness-driven, product and order-dependent
(SPOD)} weights, as introduced in \cite[eq. (3.17)]{DKLeGNS14}. 
Thus the following deterministic version of Theorem~\ref{thm:ML_abstract_lam} holds
for the error of the ML-HOQMC approximation.

\begin{theorem}
\label{thm:mlhoqmc}
Suppose that Assumption~A\ref{asm:coeff} holds with $p < q < 1$, let $L \in \N$
and for $\ell = 0, 1, \ldots, L$ let $Q_\ell^\mathrm{HO}$ be an interlaced
polynomial lattice rule as in Theorem~\ref{thm:ho-diff}.
Then the multilevel HOQMC approximation $Q^\mathrm{MLHO}_L$ 
with quadrature rule $Q^\mathrm{HO}_\ell$ on each level satisfies
\begin{equation*}
\big|\bbE_\bsy[\lambda] - Q_L^\mathrm{MLHO}(\lambda)\big|
\,\lesssim\, h_L^2 + s_L^{-2/p + 1}
+ \sum_{\ell = 0}^L N_\ell^{-1/q} \big(h_\ell^2 + s_\ell^{-1/p + 1/q}\big),
\end{equation*}
where the implied constant is independent of $h_\ell, s_\ell$ and $N_\ell$ for all $\ell = 0, 1, \ldots, L$.
\end{theorem}

Similar arguments can also be used to obtain an error bound with the same convergence rates for 
$Q_L^\mathrm{MLHO}(\calG(u))$, i.e., for the approximation of the expected value of
smooth functionals of the eigenfunction.

\section{Conclusion}
We have presented a MLQMC algorithm for approximating the expectation 
of the eigenvalue of a random elliptic EVP, and then performed
a rigorous analysis of the error. The theoretical results clearly show that
for this problem the MLQMC method exhibits better complexity than both single level 
MC/QMC and MLMC.
In the companion paper \cite{GS21b}, we will present numerical results that
also verify this superior performance of MLQMC in practice. In that paper, we will in addition
present novel ideas on how to efficiently implement the MLQMC algorithm
for EVPs.

Other interesting avenues for future research would be to consider 
non-self adjoint EVPs, e.g., convection-diffusion problems, or
to use the multi-index MC framework from, e.g., \cite{DickFeiSchw19,HajNobTemp16} to 
separate the FE and dimension truncation approximations on each level.
In principle, the algorithm studied in this paper can also be applied to
the lognormal setting as in, e.g., \cite{KSSSU17}, i.e., where each coefficient
is the exponential of a Gaussian random field, by using QMC rules for
integrals on unbounded domains.
However, in this case, the difficulty for both single level and multilevel QMC
is that the coefficients are no longer uniformly bounded
from above and below. As such, it is possible that the spectral gap,
$\lambda_2(\bsy) - \lambda_1(\bsy)$, becomes arbitrarily small
for certain parameter values. 
Since all aspects of the method (the stochastic derivative bounds, the FE error, 
the performance of the eigenvalue solver etc.)
depend inversely on the spectral gap, then both the method and the theory fail
if the gap becomes arbitrarily small.
The technique for bounding the spectral gap in \cite{GGKSS19,GGSS20} 
fails in this case because the stochastic
parameters belong to an unbounded domain.
On the other hand,
we conjecture that the spectral gap only becomes small with low probability,
and so probabilistic arguments may be able to be used to bound the gap from below.
This is again another example of the differences between stochastic EVPs and source problems,
and such analysis would make for interesting future work.

\medskip
\noindent\textbf{Acknowledgements.} 
This work is supported by the Deutsche Forschungsgemeinschaft (German Research Foundation) 
under Germany’s Excellence Strategy EXC 2181/1 - 390900948 (the Heidelberg STRUCTURES 
Excellence Cluster). 

\bibliographystyle{plain}
\bibliography{mlqmc-evp}

\begin{appendix}
\section{Proofs of recursive bounds on derivatives of the FE error}

Here, we give the proofs of the recursive bounds on the derivatives of the
FE error from Section~\ref{sec:fe-err} (Lemmas~\ref{lem:dlam_fe_rec} and \ref{lem:du-fe_rec}),
which were key to the inductive steps in the proofs of the explicit bounds in
Theorem~\ref{thm:du-fe}. Throughout we omit the $\bsx$ and $\bsy$ dependence.

\begin{proof}[Proof of Lemma~\ref{lem:dlam_fe_rec} (eigenvalue bounds)]
Let $v = v_h \in V_h$ in the variational eigenproblem \eqref{eq:var-evp},
and then subtract the FE eigenproblem \eqref{eq:fe-evp}, with the same $v_h$, 
to give the following variational relationship between the two FE errors
\begin{equation}
\label{eq:var_fe_err}
\calA(u - u_h, v_h) \,=\, 
\lambda \calM(u - u_h, v_h) 
+ (\lambda - \lambda_h)\calM(u_h, v_h),
\end{equation}
which holds for all $v_h \in V_h$.

Differentiating \eqref{eq:var_fe_err} using the Leibniz general product rule, 
gives the following recursive formula
for the $\bsnu$th derivatives of the eigenvalue and eigenfunction errors
\begin{align*}
0 \,=\, &\calA(\pdy^\bsnu (u - u_h), v_h) - \lambda \calM(\pdy^\bsnu ( u - u_h), v_h)
- (\lambda  - \lambda_h) \calM(\pdy^\bsnu u_h, v_h)
\\
&+ \sum_{j = 1}^\infty \nu_j \bigg(\int_D a_j \nabla \big[\pdy^{\bsnu - \bse_j}(u - u_h)\big] \cdot \nabla v_h 
+ \int_D b_j \big[\pdy^{\bsnu - \bse_j} (u - u_h)\big]v_h\bigg)
\\
&-\,
\sum_{\substack{\bsm \leq \bsnu \\ \bsm \neq \bsnu}} \binom{\bsnu}{\bsm}
\big[ \pdy^{\bsnu - \bsm} \lambda \calM(\pdy^\bsm(u - u_h), v_h)
+ \pdy^{\bsnu - \bsm}(\lambda - \lambda_h) \calM(\pdy^\bsm u_h, v_h)\big].
\end{align*}
Adding extra terms and using the $\calA$-orthogonality of $\Ph$,
we can write this in the following more convenient form
\begin{align}
\label{eq:rec_var}
0 \,=\, &\calA(\Ph\pdy^\bsnu (u - u_h), v_h) 
- \lambda_h \calM(\Ph\pdy^\bsnu ( u - u_h), v_h)
\nonumber\\
&- (\lambda - \lambda_h) \calM(\pdy^\bsnu u, v_h)
- \lambda_h \calM(\pdy^\bsnu u - \Ph \pdy^\bsnu u, v_h)
\nonumber\\
&+ \sum_{j = 1}^\infty \nu_j \bigg(\int_D a_j \nabla \big[\pdy^{\bsnu - \bse_j}(u - u_h)\big] \cdot \nabla v_h 
+ \int_D b_j \big[\pdy^{\bsnu - \bse_j} (u - u_h)\big]v_h\bigg)
\nonumber\\
&-\,
\sum_{\substack{\bsm \leq \bsnu \\ \bsm \neq \bsnu}} \binom{\bsnu}{\bsm}
\big[\pdy^{\bsnu - \bsm} \lambda\calM(\pdy^\bsm(u - u_h), v_h)
+ \pdy^{\bsnu - \bsm}(\lambda - \lambda_h) \calM(\pdy^\bsm u_h, v_h)\big].
\end{align}

Letting $v_h = u_h$ in \eqref{eq:rec_var} and separating out the $\bsm = \bszero$ term,  
we obtain the following formula for the derivative of the eigenvalue error
\begin{align*}
\pdy^\bsnu (\lambda &- \lambda_h) =
(\lambda_h - \lambda) \calM(\pdy^\bsnu u, u_h)
- \lambda_h \calM(\pdy^\bsnu u - \Ph \pdy^\bsnu u, u_h)
- (\pdy^\bsnu\lambda) \calM(u - u_h, u_h)
\nonumber\\
&+ \sum_{j = 1}^\infty \nu_j \bigg(\int_D a_j \nabla \big[\pdy^{\bsnu - \bse_j}(u - u_h)\big] \cdot \nabla u_h 
+ \int_D b_j \big[\pdy^{\bsnu - \bse_j} (u - u_h)\big]u_h\bigg)
\nonumber\\
&-\,
\sum_{\substack{\bszero \neq \bsm \leq \bsnu \\ \bsm \neq \bsnu}} \binom{\bsnu}{\bsm}
\big[\pdy^{\bsnu - \bsm} \lambda\calM(\pdy^\bsm(u - u_h), u_h)
+ \pdy^{\bsnu - \bsm}(\lambda - \lambda_h) \calM(\pdy^\bsm u_h, u_h)\big],
\end{align*}
where we have used the fact that $u_h$ is normalised. Also
the first two terms in \eqref{eq:rec_var} cancel because the bilinear form is symmetric and 
$(\lambda_h, u_h)$ satisfy the FE eigenvalue problem \eqref{eq:fe-evp}
with $\Ph \pdy^\bsnu (u - u_h) \in V_h$ as a test function.

Taking the absolute value, then using the triangle and Cauchy--Schwarz inequalities
gives the upper bound
\begin{align*}
|\pdy^\bsnu (\lambda &- \lambda_h)| \,\leq\,
|\lambda - \lambda_h|\, \nrm{\pdy^\bsnu u}{\calM}
 + \lambda_h \nrm{\pdy^\bsnu u - \Ph \pdy^\bsnu u}{\calM}
+  |\pdy^\bsnu\lambda|\, \nrm{u - u_h}{\calM}
\\
&+ \sum_{j = 1}^\infty \nu_j \big[\nrm{a_j}{L^\infty} 
\nrm{\pdy^{\bsnu - \bse_j}(u - u_h)}{V} \nrm{u_h}{V}
+\nrm{b_j}{L^\infty} \nrm{\pdy^{\bsnu - \bse_j} (u - u_h)}{L^2}\nrm{u_h}{L^2}\big]
\\
&+\,
\sum_{\substack{\bszero \neq \bsm \leq \bsnu \\ \bsm \neq \bsnu}} \binom{\bsnu}{\bsm}
\big[|\pdy^{\bsnu - \bsm} \lambda| \,\nrm{\pdy^\bsm(u - u_h)}{\calM}
+ |\pdy^{\bsnu - \bsm}(\lambda - \lambda_h)| \,\nrm{\pdy^\bsm u_h}{\calM}\big],
\end{align*}
where we have again simplified by using $\nrm{u_h}{\calM} = 1$.
Then, using the equivalence of norms \eqref{eq:M_equiv} and the Poincar\'e
inequality \eqref{eq:poin}, we can bound the $\calM$- and $L^2$-norms
by the corresponding $V$-norms, to give
\begin{align*}
|\pdy^\bsnu (\lambda &- \lambda_h)| \,\leq\,
\sqrt{\frac{\amax}{\evalueLap}}\Big[|\lambda - \lambda_h| \,\nrm{\pdy^\bsnu u}{V} 
 + \lambdabar \nrm{\pdy^\bsnu u - \Ph \pdy^\bsnu u}{V}
+  |\pdy^\bsnu\lambda| \,\nrm{u - u_h}{V}\Big]
\\
&+ \ubar\bigg(1 + \frac{1}{\evalueLap}\bigg)\sum_{j = 1}^\infty \nu_j 
\frac{\beta_j}{C_\bsbeta}
\nrm{\pdy^{\bsnu - \bse_j}(u - u_h)}{V}
\\
&+\,
\sqrt{\frac{\amax}{\evalueLap}}\sum_{\substack{\bszero \neq \bsm \leq \bsnu \\ \bsm \neq \bsnu}} \binom{\bsnu}{\bsm}
\big[|\pdy^{\bsnu - \bsm} \lambda| \nrm{\pdy^\bsm(u - u_h)}{V}
+ |\pdy^{\bsnu - \bsm}(\lambda - \lambda_h)| \nrm{\pdy^\bsm u_h}{V}\big],
\end{align*}
where we have also used the upper bounds \eqref{eq:lam_bnd} and \eqref{eq:u_bnd},
and the definition of $\beta_j$ \eqref{eq:beta}.

Substituting in the upper bounds on the derivatives \eqref{eq:dlambda} and \eqref{eq:du_V},
the bound on the projection error \eqref{eq:du_proj_err}, and then
the bounds on the FE errors \eqref{eq:fe_lam} and \eqref{eq:fe_u},
we have the upper bound
\begin{align*}
|\pdy^\bsnu (\lambda - &\lambda_h)| \,\leq\,
\sqrt{\frac{\amax}{\evalueLap}}\big[C_{\lambda} h
\ubar\bsbeta^\bsnu+
\lambdabar C_{\calP}\overline{\bsbeta}^\bsnu + \overline{\lambda}C_{u}\bsbeta^\bsnu
\Big] h |\bsnu|!^{1 + \upepsilon}
\\ &
+ \frac{\ubar}{C_\bsbeta}\bigg(1 + \frac{1}{\evalueLap}\bigg)\sum_{j = 1}^\infty \nu_j 
\beta_j \nrm{\pdy^{\bsnu - \bse_j}(u - u_h)}{V}
+ \sqrt{\frac{\amax}{\evalueLap}}\sum_{\substack{\bszero \neq \bsm \leq \bsnu \\ \bsm \neq \bsnu}} \binom{\bsnu}{\bsm}
\\ & 
\cdot\Big[\overline{\lambda} |\bsnu - \bsm|!^{1 + \upepsilon} \bsbeta^{\bsnu - \bsm}
\nrm{\pdy^\bsm(u - u_h)}{V}
+ \overline{u}|\bsm|!^{1 + \upepsilon} \bsbeta^\bsm 
|\pdy^{\bsnu - \bsm}(\lambda - \lambda_h)|\Big].
\end{align*}
Note that we can simplify the sum on the last line using the
symmetry of the binomial coefficient, $\binom{n}{k} = \binom{n}{n - k}$,
as follows. First, we separate it into two sums
\begin{align}
\label{eq:sum_simplify}
\sum_{\substack{\bszero \neq \bsm \leq \bsnu \\ \bsm \neq \bsnu}} \binom{\bsnu}{\bsm}&
\Big[\overline{\lambda} |\bsnu - \bsm|!^{1 + \upepsilon} \bsbeta^{\bsnu - \bsm}
\nrm{\pdy^\bsm(u - u_h)}{V}
+ \overline{u}|\bsm|!^{1 + \upepsilon} \bsbeta^\bsm |\pdy^{\bsnu - \bsm}(\lambda - \lambda_h)|\Big]
\nonumber\\
=\,&
\lambdabar\sum_{\substack{\bszero \neq \bsm \leq \bsnu \\ \bsm \neq \bsnu}} \binom{\bsnu}{\bsm} |\bsnu - \bsm|!^{1 + \upepsilon} \bsbeta^{\bsnu - \bsm}
\nrm{\pdy^\bsm(u - u_h)}{V}
\nonumber\\&
+ \overline{u}
\sum_{\substack{\bszero \neq \bsm \leq \bsnu \\ \bsm \neq \bsnu}} \binom{\bsnu}{\bsm}
|\bsm|!^{1 + \upepsilon} \bsbeta^\bsm |\pdy^{\bsnu - \bsm}(\lambda - \lambda_h)|\Big]
\nonumber\\
=\,&
(\lambdabar + \overline{u})
\sum_{\substack{\bszero \neq \bsm \leq \bsnu \\ \bsm \neq \bsnu}} 
\binom{\bsnu}{\bsm} |\bsm|!^{1 + \upepsilon} \bsbeta^{\bsm}
\Big[\nrm{\pdy^{\bsnu - \bsm}(u - u_h)}{V}
+ |\pdy^{\bsnu - \bsm}(\lambda - \lambda_h)|\Big],
\end{align}
where to obtain the last equality we have simply relabelled the indices in the first sum.

Then, since $\beta_j \leq \betabar_j$ and $h$ is sufficiently small 
(i.e., $h \leq \hsuff$ with $\hsuff$ as in \eqref{eq:hbar}), the result 
\eqref{eq:dlam_fe_rec_a} holds. The constant is given by
\[
\CI \,\coloneqq\,
\max\bigg\{\sqrt{\frac{\amax}{\evalueLap}}\big[\ubar\hsuff C_{\lambda} +
\lambdabar(C_{\calP} + C_{u})
\big],\;
\frac{ \ubar}{C_\bsbeta}\bigg(1 + \frac{1}{\evalueLap}\bigg) ,\;
\sqrt{\frac{\amax}{\evalueLap}}(\lambdabar + \ubar)
\bigg\},
\]
which is  independent of $\bsy$, $h$ and $\bsnu$.

For the second result \eqref{eq:dlam_fe_rec_b}, 
using \cite[Lemma~3.1]{BO89} the eigenvalue error can also be written as
\[
\lambda - \lambda_h \,=\, - \calA(u - u_h, u - u_h) + \lambda\calM( u - u_h, u - u_h),
\]
which after taking the $\bsnu$th derivative becomes
\begin{align*}
\pdy^\bsnu(\lambda &- \lambda_h) \,=\,
- \sum_{\bsm \leq\bsnu} \binom{\bsnu}{\bsm}\calA(\pdy^{\bsnu - \bsm}(u - u_h), \pdy^\bsm(u - u_h))
\\
&- \sum_{j = 1}^\infty\sum_{\bsm \leq \bsnu - \bse_j} \nu_j \binom{\bsnu - \bse_j}{\bsm}
\bigg[\int_D a_j \nabla \pdy^{\bsnu - \bse_j - \bsm}(u - u_h)
\cdot \nabla \pdy^\bsm (u - u_h)
\\
&+ \int_D b_j \pdy^{\bsnu - \bse_j - \bsm} (u - u_h) \pdy^\bsm(u - u_h) \bigg]
\\
&+ \sum_{\bsm \leq \bsnu} \sum_{\bsk \leq \bsm}
\binom{\bsnu}{\bsm} \binom{\bsm}{\bsk} \pdy^{\bsnu - \bsm}\lambda
\calM(\pdy^{\bsm - \bsk}(u - u_h), \pdy^\bsk(u - u_h)).
\end{align*}

Taking the absolute value, then using the triangle, Cauchy--Schwarz 
and Poincar\'e \eqref{eq:poin} inequalities, along with the norm equivalences 
\eqref{eq:A_equiv}, \eqref{eq:M_equiv}, gives
\begin{align*}
&|\pdy^\bsnu(\lambda - \lambda_h)| \,\leq\,
\amax\bigg( 1 + \frac{1}{\evalueLap}\bigg)\sum_{\bsm \leq\bsnu} \binom{\bsnu}{\bsm}
\nrm{\pdy^{\bsnu - \bsm}(u - u_h)}{V} \nrm{\pdy^\bsm(u - u_h)}{V}
\\
&+ \bigg( 1 + \frac{1}{\evalueLap}\bigg) \sum_{j = 1}^\infty\sum_{\bsm \leq \bsnu - \bse_j} 
\nu_j \binom{\bsnu - \bse_j}{\bsm} \frac{\beta_j}{C_\bsbeta}
\nrm{ \pdy^{\bsnu - \bse_j - \bsm}(u - u_h)}{V}
\nrm{\pdy^\bsm (u - u_h)}{V}
\\
&+ \frac{\amax}{\evalueLap}\sum_{\bsm \leq \bsnu} \sum_{\bsk \leq \bsm}
\binom{\bsnu}{\bsm} \binom{\bsm}{\bsk} 
|\pdy^{\bsnu - \bsm} \lambda|
\nrm{\pdy^{\bsm - \bsk}(u - u_h)}{V} \nrm{\pdy^\bsk(u - u_h))}{V}.
\end{align*}
Finally, substituting in the upper bound \eqref{eq:dlambda} on the derivative of $\lambda$
gives the desired result \eqref{eq:dlam_fe_rec_b}. The constant is given by
\[
\CII \,\coloneqq\,
\big[1/C_\bsbeta + \amax(1 + \lambdabar)\big]\bigg(1 + \frac{1}{\evalueLap}\bigg),
\]
which is independent of $h$, $\bsy$ and $\bsnu$.
\end{proof}

\begin{proof}[Proof of Lemma~\ref{lem:du-fe_rec} (eigenfunction bound)]
We deal with the eigenfunction error
projected onto $V_h$, as opposed to $\pdy^\bsnu(u - u_h)$, because the latter belongs to
$V$ but not to $V_h$. 
As such, we first separate the error as
\begin{align}
\label{eq:du_proj_sep}
\nrm{\pdy^\bsnu(u - u_h)}{V} \,&\leq\,
\nrm{\Ph\pdy^\bsnu(u - u_h)}{V}  + \nrm{\pdy^\bsnu u -  \Ph\pdy^\bsnu u}{V}
\nonumber\\
&\leq\, 
\nrm{\Ph\pdy^\bsnu(u - u_h)}{V} + C_{\calP} h |\bsnu|!^{1 + \upepsilon} \bsbetabar^\bsnu
,
\end{align}
where in the second inequality we have used the bound \eqref{eq:du_proj_err}.

Similar to the proof of \cite[Lemma~3.4]{GGKSS19}, 
the bilinear form that acts on $\Ph \pdy^\bsnu (u - u_h)$ (namely,
$\calA -\lambda_h \calM$) is only coercive on the orthogonal complement
of the eigenspace corresponding to $\lambda_h$, which 
we denote by $E(\lambda_h)^\perp$. Hence, 
to obtain the recursive formula
for the derivative of the eigenfunction error, we first make the following
orthogonal decomposition.
The FE eigenfunctions form an orthogonal basis for $V_h$, and so we have
\begin{equation}
\label{eq:du_decomp}
\Ph \pdy^\bsnu (u - u_h) \,=\, \calM(\Ph \pdy^\bsnu (u - u_h), u_h) u_h + \varphi_h,
\end{equation}
where $\varphi_h \in E(\lambda_h)^\perp$. Then we can bound the norm by
\begin{equation}
\label{eq:du_nrm_decomp}
\nrm{\Ph \pdy^\bsnu (u - u_h)}{V} \,\leq\, 
|\calM(\Ph \pdy^\bsnu (u - u_h), u_h)| \nrm{u_h}{V}
+ \nrm{\varphi_h}{V}.
\end{equation}

To bound the first term in this decomposition \eqref{eq:du_nrm_decomp}, 
first observe that we can write
\begin{align}
\nonumber
\label{eq:M_Ph_du-du_h}
|\calM(\Ph \pdy^\bsnu (u - u_h), u_h)|
\,\leq\, &|\calM(\pdy^\bsnu u, u) - \calM(\pdy^\bsnu u_h, u_h)|\\
&+ |\calM(\pdy^\bsnu u, u - u_h)| + |\calM(\pdy^\bsnu u - \Ph \pdy^\bsnu u, u_h)|.
\end{align}

The first term on the right in \eqref{eq:M_Ph_du-du_h} 
can be bounded by differentiating the normalisation
equations $\nrm{u}{\calM} = 1$ and $\nrm{u_h}{\calM} = 1$ 
(see \cite[eq. (3.15)]{GGKSS19}) to give
\begin{align*}
\calM(\pdy^\bsnu u, &u) - \calM(\pdy^\bsnu u_h, u_h)
\\
\,&=\, -\frac{1}{2} \sum_{\substack{\bszero \neq \bsm \leq \bsnu \\ \bsm \neq \bsnu }}
\binom{\bsnu}{\bsm}
\big[\calM(\pdy^{\bsnu - \bsm}u, \pdy^\bsm u) - \calM(\pdy^{\bsnu - \bsm}u_h, \pdy^\bsm u_h)\big]
\\
&=\, -\frac{1}{2} \sum_{\substack{\bszero \neq \bsm \leq \bsnu \\ \bsm \neq \bsnu }}
\binom{\bsnu}{\bsm} \big[
\calM(\pdy^{\bsnu - \bsm}(u - u_h), \pdy^\bsm u) + \calM(\pdy^{\bsnu - \bsm}u_h, \pdy^\bsm (u - u_h))
\big].
\end{align*}
Then, using the triangle inequality, the Cauchy--Schwarz inequality, 
the equivalence of norms \eqref{eq:M_equiv} and the Poincar\'e  inequality \eqref{eq:poin},
gives the upper bound
\begin{align*}
|\calM&( \pdy^\bsnu u, u) - \calM(\pdy^\bsnu u_h, u_h)|
\\
&\leq\, 
\frac{\amax}{2\evalueLap} \sum_{\substack{\bszero \neq \bsm \leq \bsnu \\ \bsm \neq \bsnu }}
\binom{\bsnu}{\bsm} \big[\nrm{\pdy^{\bsnu - \bsm}(u - u_h)}{V} \nrm{\pdy^\bsm u}{V}
+ \nrm{\pdy^{\bsnu - \bsm}u_h}{V} \nrm{\pdy^\bsm (u - u_h)}{V}\big]
\\
&\leq\, 
\ubar\frac{\amax}{2\evalueLap} \sum_{\substack{\bszero \neq \bsm \leq \bsnu \\ \bsm \neq \bsnu }}
\binom{\bsnu}{\bsm} \big[
|\bsm|!^{1 + \upepsilon} \bsbeta^\bsm \nrm{\pdy^{\bsnu - \bsm}(u - u_h)}{V} 
\\
&\qquad\qquad\qquad
+ |\bsnu - \bsm|!^{1 + \upepsilon} \bsbeta^{\bsnu -\bsm}\nrm{\pdy^\bsm (u - u_h)}{V}
\big]
\\
&
=\, 
\ubar\frac{\amax}{\evalueLap} \sum_{\substack{\bszero \neq \bsm \leq \bsnu \\ \bsm \neq \bsnu }}
\binom{\bsnu}{\bsm} 
|\bsm|!^{1 + \upepsilon} \bsbeta^\bsm \nrm{\pdy^{\bsnu - \bsm}(u - u_h)}{V} ,
\end{align*}
where for the second last inequality we have used the upper bound \eqref{eq:du_V} and
the analogous bound for $u_h$.
For the equality on the last line,  we have simplified the sum using the symmetry of the 
binomial coefficient as in \eqref{eq:sum_simplify}.

To bound the second and third terms in \eqref{eq:M_Ph_du-du_h} 
we use the Cauchy--Schwarz inequality, 
the equivalence of norms \eqref{eq:M_equiv}, and
the Poincar\'e inequality \eqref{eq:poin}, followed by the bound on the projection error
\eqref{eq:du_proj_err} and the bound on the FE error \eqref{eq:fe_u}, which gives
\begin{align}
\label{eq:M_of_Phdu_0}
|\calM(\pdy^\bsnu u, u - u_h)| &+ |\calM(\pdy^\bsnu u - \Ph \pdy^\bsnu u, u_h)|
\nonumber\\
\,&\leq\, \nrm{\pdy^\bsnu u}{\calM} \nrm{u - u_h}{\calM} 
+ \nrm{\pdy^\bsnu u - \Ph \pdy^\bsnu u}{\calM}
\nonumber\\
&\leq\, \frac{\amax}{\evalueLap}\nrm{\pdy^\bsnu u}{V} \nrm{u - u_h}{V} 
+ \sqrt{\frac{\amax}{\evalueLap}}\nrm{\pdy^\bsnu u - \Ph \pdy^\bsnu u}{V}
\nonumber\\
&\leq\,
\frac{\amax}{\evalueLap} C_{u} h |\bsnu|!^{1 + \upepsilon} \bsbeta^\bsnu
+ \sqrt{\frac{\amax}{\evalueLap}} C_{\calP} h |\bsnu|!^{1 + \upepsilon} \bsbetabar^\bsnu
\nonumber\\
&\leq\, 
 \bigg(\frac{\amax}{\evalueLap} C_{u} + \sqrt{\frac{\amax}{\evalueLap}}C_{\calP} \bigg)
h |\bsnu|!^{1 + \upepsilon} \bsbetabar^\bsnu,
\end{align}
where in the last inequality we have used that $\beta_j \leq \betabar_j$.

Substituting these two bounds into \eqref{eq:M_Ph_du-du_h} then multiplying by 
$\|u_h\|_V$ gives the following upper bound
on the first term of the decomposition \eqref{eq:du_nrm_decomp} 
\begin{align}
\label{eq:M_of_Phdu}
\nonumber
|\calM(\Ph \pdy^\bsnu (u - u_h), u_h)|\,\|u_h\|_V
\,\leq\,&
\ubar \bigg(\frac{\amax}{\evalueLap} C_{u} + \sqrt{\frac{\amax}{\evalueLap}}C_{\calP} \bigg)
h |\bsnu|!^{1 + \upepsilon} \bsbetabar^\bsnu
\\
&+ \ubar^2\frac{\amax}{\evalueLap} \sum_{\substack{\bszero \neq \bsm \leq \bsnu \\ \bsm \neq \bsnu }}
 |\bsm|!^{1 + \upepsilon} \bsbeta^\bsm \nrm{\pdy^{\bsnu - \bsm}(u - u_h)}{V} .
\end{align}
Note that we have also used \eqref{eq:u_bnd}.

Next, to bound the norm of $\varphi_h$ (the second term in the decomposition 
\eqref{eq:du_nrm_decomp}), we let $v_h = \varphi_h$ in \eqref{eq:rec_var}
and then rearrange the terms to give
\begin{align}
\label{eq:phi_h_0}
\calA(&\Ph \pdy^\bsnu (u - u_h), \varphi_h) - \lambda_h \calM(\Ph \pdy^\bsnu (u - u_h), \varphi_h)
\,=\,
(\lambda  - \lambda_h) \calM(\pdy^\bsnu u, \varphi_h)
\nonumber\\
&+ \lambda_h \calM(\pdy^\bsnu u - \Ph \pdy^\bsnu u, \varphi_h)
+ \pdy^\bsnu \lambda \calM( u - u_h, \varphi_h) 
+ \pdy^\bsnu( \lambda - \lambda_h) \calM( u_h, \varphi_h)
\nonumber\\
&- \sum_{j = 1}^\infty \nu_j \bigg(\int_D a_j \nabla \big[\pdy^{\bsnu - \bse_j}(u - u_h)\big] \cdot 
\nabla \varphi_h
+ \int_D b_j \big[\pdy^{\bsnu - \bse_j} (u - u_h)\big]\varphi_h\bigg)
\\\nonumber
&+\,
\sum_{\substack{\bszero \neq \bsm \leq \bsnu \\ \bsm \neq \bsnu}} \binom{\bsnu}{\bsm}
\big[\pdy^{\bsnu - \bsm} \lambda\calM(\pdy^\bsm(u - u_h), \varphi_h)
+ \pdy^{\bsnu - \bsm}(\lambda - \lambda_h) \calM(\pdy^\bsm u_h, \varphi_h)\big].
\end{align}

Again using the decomposition \eqref{eq:du_decomp} and the fact that 
$u_h$ satisfies the eigenproblem \eqref{eq:fe-evp} with $\varphi_h \in V_h$
as a test function, the left hand side of \eqref{eq:phi_h_0}
simplifies to $\calA(\varphi_h, \varphi_h) - \lambda_h \calM(\varphi_h, \varphi_h)$.
Since $\varphi_h \in E(\lambda_h)^\perp$, we can use the FE version of the coercivity estimate
\cite[Lemma~3.1]{GGKSS19} (see also Remark~3.2 that follows), to bound this from below by
\begin{equation}
\label{eq:phi_h_lower}
 \calA(\varphi_h, \varphi_h) - \lambda_h \calM(\varphi_h, \varphi_h)
 \,\geq\, \amin \bigg(\frac{\lambda_{2, h} - \lambda_h}{\lambda_{2, h}} \bigg)
 \nrm{\varphi_h}{V}^2
 \,\geq\, \frac{\amin \rho}{ 2\overline{\lambda_2}} \nrm{\varphi_h}{V}^2,
\end{equation}
where in the last inequality we have used the upper bound \eqref{eq:lam_bnd},
along the lower bound \eqref{eq:gap_h} on the FE spectral gap,
which is applicable for $h$ sufficiently small.

Taking the absolute value, the right hand side of \eqref{eq:phi_h_0} can be
bounded using the triangle inequality, the Cauchy--Schwarz inequality,
the equivalence of norms \eqref{eq:M_equiv}, and the Poincar\'e inequality \eqref{eq:poin},
which, combined with the lower bound  \eqref{eq:phi_h_lower}, gives
\begin{align*}
&\frac{\amin \rho}{ 2\overline{\lambda_2}} \nrm{\varphi_h}{V}^2 \,\leq\, 
\frac{\amax}{\evalueLap} \Big(|\lambda  - \lambda_h| \nrm{\pdy^\bsnu u}{V} 
+\lambdabar\nrm{\pdy^\bsnu u - \Ph \pdy^\bsnu u}{V} \\
&+ |\pdy^\bsnu \lambda| \nrm{u - u_h}{V} 
+ |\pdy^\bsnu( \lambda - \lambda_h)| \nrm{u_h}{V} \Big) \|\varphi_h\|_{V}
\\
&+ \sum_{j = 1}^\infty \nu_j \bigg(\nrm{a_j}{L^\infty} + \frac{1}{\evalueLap} \nrm{b_j}{L^\infty} \bigg) 
\nrm{\pdy^{\bsnu - \bse_j} (u - u_h)}{V} \nrm{\varphi_h}{V}
\\
&+\,
\frac{\amax}{\evalueLap}
\sum_{\substack{\bszero \neq\bsm \leq \bsnu \\ \bsm \neq \bsnu}}
\binom{\bsnu}{\bsm}
\big(|\pdy^{\bsnu - \bsm} \lambda | \nrm{\pdy^\bsm(u - u_h)}{V}
+ |\pdy^{\bsnu - \bsm}(\lambda - \lambda_h)| \nrm{\pdy^\bsm u_h}{V}\big) \nrm{ \varphi_h}{V}.
\end{align*}
Dividing through by $\amin \rho/(2 \overline{\lambda_2})\nrm{\varphi_h}{V}$,
then using the bounds \eqref{eq:fe_lam}, \eqref{eq:dlambda}, \eqref{eq:du_V}, \eqref{eq:du_proj_err}, along with the fact that that $\beta_j \leq \betabar_j$ for all $j \in \N$
and $h \leq \hsuff$, 
we have that the norm of $\varphi_h$ is bounded by
\begin{align}
\label{eq:phi_h_bound}
&\nrm{\varphi_h}{V} \,\leq\, 
\frac{\amax}{\amin} \frac{2\overline{\lambda_2}}{\rho\evalueLap} 
\Big[\ubar \, \hsuff C_{\lambda}    + \lambdabar C_{\calP} + \lambdabar C_{u} \Big]
h |\bsnu|!^{1 + \upepsilon} \bsbetabar^\bsnu 
+ \frac{\amax}{\amin}\frac{2\ubar\overline{\lambda_2}}{\rho\evalueLap} 
|\pdy^\bsnu( \lambda - \lambda_h)|
\\\nonumber
&+ \frac{2\overline{\lambda_2}}{\amin \rho C_\bsbeta} \bigg(1 + \frac{1}{\evalueLap}\bigg)
\sum_{j = 1}^\infty \nu_j \beta_j
\nrm{\pdy^{\bsnu - \bse_j} (u - u_h)}{V}
\\\nonumber
&+\,
\frac{\amax}{\amin}\frac{2\overline{\lambda_2}}{\rho\evalueLap} 
(\lambdabar + \ubar)
\sum_{\substack{\bszero \neq \bsm \leq \bsnu \\ \bsm \neq \bsnu}} \binom{\bsnu}{\bsm} 
|\bsm|!^{1 + \upepsilon} 
\bsbeta^\bsm 
\Big[\nrm{\pdy^{\bsnu - \bsm}(u - u_h)}{V}+ 
\|\pdy^{\bsnu - \bsm}(\lambda - \lambda_h)| \Big].
\end{align}
Note that to get the sum on the last line we have again simplified 
similarly to \eqref{eq:sum_simplify}.

Substituting the bounds \eqref{eq:du_nrm_decomp}, followed by \eqref{eq:M_of_Phdu}
and \eqref{eq:phi_h_bound}, into the decomposition \eqref{eq:du_proj_sep}
gives the following recursive bound on the derivative of the eigenfunction error
\begin{align*}
&\nrm{\pdy^\bsnu (u - u_h)}{V} 
\\
&\leq\, 
\bigg[\frac{\amax}{\amin} \frac{2\overline{\lambda_2}}{\rho\evalueLap} 
\big(\ubar \hsuff C_{\lambda}  + \lambdabar C_{\calP} 
+ \lambdabar C_{u} \big)
+ \ubar\bigg(\frac{\amax}{\evalueLap} C_{u} + \sqrt{\frac{\amax}{\evalueLap}}C_{\calP} \bigg)
 + C_{\calP}\bigg]
h |\bsnu|!^{1 + \upepsilon} \bsbetabar^\bsnu
\\
&+ \frac{\amax}{\amin} \frac{2\ubar\overline{\lambda_2}}{\rho\evalueLap}
 |\pdy^\bsnu (\lambda - \lambda_h)| 
+ \frac{2\overline{\lambda_2}}{\amin \rho C_\bsbeta} \bigg(1 + \frac{1}{\evalueLap}\bigg)
\sum_{j = 1}^\infty \nu_j \beta_j
\nrm{\pdy^{\bsnu - \bse_j} (u - u_h)}{V}
\\
&+\,
\frac{\amax}{\amin} \frac{2\overline{\lambda_2}}{\rho\evalueLap} (\lambdabar + \ubar)
\sum_{\substack{\bszero \neq \bsm \leq \bsnu \\ \bsm \neq \bsnu}}
\binom{\bsnu}{\bsm} |\bsm|!^{1 + \upepsilon} \bsbeta^\bsm 
\Big[\nrm{\pdy^{\bsnu - \bsm}(u - u_h)}{V}+ 
\|\pdy^{\bsnu - \bsm}(\lambda - \lambda_h)| \Big]
\\ \nonumber
&+ \ubar^2\frac{\amax}{\evalueLap} 
\sum_{\substack{\bszero \neq \bsm \leq \bsnu \\ \bsm \neq \bsnu }}
\binom{\bsnu}{\bsm}
 |\bsm|!^{1 + \upepsilon} \bsbeta^\bsm \nrm{\pdy^{\bsnu - \bsm}(u - u_h)}{V} .
\end{align*}
Observe that all of the constant terms are independent of $\bsy$, $h$ and $\bsnu$.

To obtain the final result with a right hand side that does not depend on any derivative
of order $\bsnu$, we now substitute the recursive formula \eqref{eq:dlam_fe_rec_a} for 
$\pdy^\bsnu(\lambda - \lambda_h)$. 
After grouping the similar terms and collecting all of the constants into $\CIII$ 
we have the final result. Since $\CI$ from \eqref{eq:dlam_fe_rec_a} and 
all of the constants above are independent of $\bsy$, $h$, and $\bsnu$,
the final constant $\CIII$ is as well.
\end{proof}

\end{appendix}

\end{document}